\documentclass[a4paper,reqno,11pt]{amsart}
\usepackage{amssymb,verbatim}
\usepackage{a4wide}
\usepackage[T1]{fontenc}
\usepackage{mathtools}
\usepackage{ifpdf}
\usepackage{tikz,ifthen}
\usepackage{multirow}
\usepackage{breqn}
\usepackage{paralist}
\usepackage{graphicx}
\usepackage{subfigure}
\usepackage[all]{xy}
\usepackage{caption}

\usepackage{enumitem}
\usepackage{xcolor}
 \definecolor{darkred}{rgb}{0.4,0,0}
 \definecolor{darkgreen}{rgb}{0,0.4,0}
 \definecolor{darkblue}{rgb}{0,0,0.4}
\usepackage[colorlinks,citecolor=darkred,linkcolor=darkblue,urlcolor=darkblue,filecolor=darkblue]{hyperref}
\usepackage[noabbrev, nameinlink]{cleveref}
\usepackage[url=false,giveninits=true,backend=biber,style=alphabetic]{biblatex}
\DeclareFieldFormat{title}{\mkbibquote{#1}}
\AtEveryBibitem{\clearfield{issn}}
\AtEveryCitekey{\clearfield{issn}}
\renewbibmacro{in:}{%
 \ifentrytype{article}{}{%
	\printtext{\bibstring{in}\intitlepunct}%
 }%
}
\DeclareFieldFormat[misc]{date}{Preprint (#1)}
\DeclareFieldFormat[misc]{doi}{}
 \usepackage[normalem]{ulem}
 \newcommand\redsout{\bgroup\markoverwith{\textcolor{red}{\rule[0.5ex]{2pt}{0.8pt}}}\ULon}

\newcommand{\RR}{\mathbb{R}}
\newcommand{\N}{\mathbf{N}}
\newcommand{\CC}{\mathbb{C}}
\newcommand{\QQ}{\mathbb{Q}}

\renewcommand{\H}{\mathcal{H}}
\newcommand{\HH}{\mathbb{H}}
\newcommand{\ZZ}{\mathbb{Z}}
\newcommand{\NN}{\mathbb{N}}
\newcommand{\SL}{\operatorname{SL}}
\newcommand{\GL}{\operatorname{GL}}
\newcommand{\PSL}{\operatorname{PSL}}

\newcommand{\PSO}{\operatorname{PSO}}
\newcommand{\Maff}{\mathcal N}
\newcommand{\pr}{\mathrm{pr}}
\newcommand{\esq}{\mathcal{E}_{\scalebox{.382}{$\blacksquare$}}}
\renewcommand{\Re}{\operatorname{Re}}
\renewcommand{\Im}{\operatorname{Im}}
\newcommand{\diag}{\operatorname{diag}}

\newcommand{\ann}{\operatorname{Ann}}
\newcommand{\Aff}{\operatorname{Aff}}
\newcommand{\Aut}{\operatorname{Aut}}
\newcommand{\stab}{\operatorname{Stab}}
\newcommand{\tr}{\operatorname{Tr}}

\usepackage{xspace}
\def\eg{\textit{e.g.}\xspace}
\def\ie{\textit{i.e.}\xspace}
\def\cf{\textit{cf.}\xspace}

\theoremstyle{plain}
\newtheorem{Theorem}{Theorem}[section]
\newtheorem{Corollary}[Theorem]{Corollary}
\newtheorem{Proposition}[Theorem]{Proposition}
\newtheorem{Lemma}[Theorem]{Lemma}

\newtheorem*{NoNumberClaim}{Claim}

\theoremstyle{definition}
\newtheorem{Remark}[Theorem]{Remark}

\begin{document}

\title[Diffusion rates for the wind-tree model]
{Diffusion rate in non-generic directions in the wind-tree model}

\date{\today}
\author{Sylvain Crovisier}
\address{Laboratoire de Mathématiques d'Orsay, CNRS - UMR 8628,
Université Paris-Saclay
91405 Orsay Cedex, France 
}
\email{Sylvain.Crovisier@universite-paris-saclay.fr}

\author{Pascal Hubert}
\address{Aix Marseille Universit\'e, CNRS, Centrale Marseille, Institut de Math\' ematiques de Marseille, I2M - UMR 7373\\13453 Marseille, France.}
\email{pascal.hubert@univ-amu.fr}

\author{Erwan Lanneau}
\address{
UMR CNRS 5582,
Univ. Grenoble Alpes, CNRS, Institut Fourier, F-38000 Grenoble, France}
\email{erwan.lanneau@univ-grenoble-alpes.fr}

\author{Angel Pardo}
\address{
Departamento de Matemática y Ciencia de la Computación,
Universidad de Santiago de Chile,
Las Sophoras 173, Estación Central, Santiago, Chile..
}
\email{angel.pardo@usach.cl}


\begin{abstract}
We show that any real number in $[0,1)$ is a diffusion rate for the wind-tree model with rational parameters. We will also provide a criterion in order to describe the shape of the Lyapunov spectrum of cocycles obtained as suspension of a representation. As an application, we exhibit an infinite family of wind-tree billiards for which the interior of the Lyapunov spectrum is a big as possible: this is the full square $(0,1)^2$. 
To the best of the knowledge of the authors, these are the first complete description where the interior of the Lyapunov spectrum is known explicitly in dimension two, even for general Fuchsian groups.
\end{abstract}

\maketitle

\section{Introduction}

The wind-tree model
(in a slightly different version)
was introduced by P.~Ehrenfest and
T.~Ehrenfest~\cite{Ehrenfest} about a century ago and investigated later by
J. Hardy and J.~Weber~\cite{Hardy:Weber}. All of these
studies had physical motivations.

More concretely, the model $T(a,b)$ is obtained by putting a fixed $[0,a]\times [0,b]$ rectangular scatterer
$\ZZ^2$-periodically in the plane with sides parallel to the axes and examines the
behavior of the trajectory of $x\in T(a,b)$ which follows the rules of elastic collision when it hits one of the rectangles, as in \Cref{fig:wtm}.

\begin{figure}[hbt]
\includegraphics{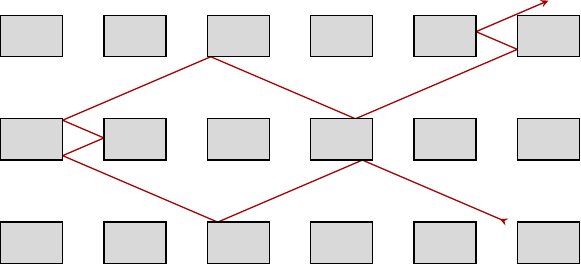}
\caption{The (periodic) wind-tree model.}
\label{fig:wtm}
\end{figure}

Several advances were obtained recently using the powerful technology
of deviation spectrum of measured foliations on surfaces, and the
underlying dynamics in the moduli space.
In this work we focus on the
diffusion rate, namely
$$
\delta_\theta(x) = \limsup_{T \to +\infty} \frac{\log d(x,\phi^\theta_T(x))}{\log T} \in [0,1],
$$
where $\phi^\theta_{T}(x)$ is the position of a particle
after time $T$ (with constant velocity) starting from position $x$ in direction $\theta\in\mathbb S^1$ and
$d(.,.)$ is the Euclidean distance on $\RR^2$.
In the sequel, we consider the
set $\mathcal D\subset \mathbb S^1$ of directions $\theta$ whose
$\delta_\theta(x)$ is constant for almost every $x\in T(a,b)$. The latter will be simply denoted $\delta_\theta$.

Delecroix--Hubert--Lelièvre~\cite{Delecroix:Hubert:Lelievre} proved that
$\mathcal D$ has full Lebesgue measure in $\mathbb S^1$ and that
$\delta_\theta=\frac{2}{3}$ for almost every $\theta$ and every
$(a,b)\in\mathcal E$, where
$$
\mathcal E=\{(a,b) \in(0,1)^2 ;\, 1/(1-a) = x + y \sqrt{D},\ 1/(1-b) = (1-x) + y \sqrt{D}, x,y \in \mathbb Q,\
\textrm{$D\in \mathbb Z_{>0}$}\}.
$$
In particular, $(0,1)^2 \cap \QQ^2 \subset \mathcal E$.
They also showed that this property holds for almost every parameter $(a,b)\in(0,1)^2$.
Later,~\cite{Delecroix:Hubert:Lelievre} strengthen the result to any parameter, using an argument by Chaika and Eskin~\cite{Chaika:Eskin}:
the diffusion rate does not depend either on the concrete values $(a,b)$ of the obstacle nor on the choice of a generic direction $\theta$ on the circle.
In this paper we establish the following result, providing a positive answer to a question of A.~Zorich~\cite{Zorich:discussion}.
\begin{Theorem}
\label{thm:main}
For any $(a,b)\in \mathcal E$ one has
$$
\left\{\delta_\theta;\, \theta\in \mathcal D \right\} = [0,1].
$$
\end{Theorem}

\begin{Remark}
Our methods allows to establish the corresponding result in the case of the Delecroix--Zorich variant~\cite{Delecroix:Zorich}.
More precisely, a Delecroix--Zorich wind-tree model consists of identical connected vertically and horizontally symmetric right-angled scatterers arranged $\ZZ^2$-periodically in the plane with the sides parallel to the axes, as in \Cref{fig:DZ-wtm}.
\begin{figure}[hbt]
\includegraphics{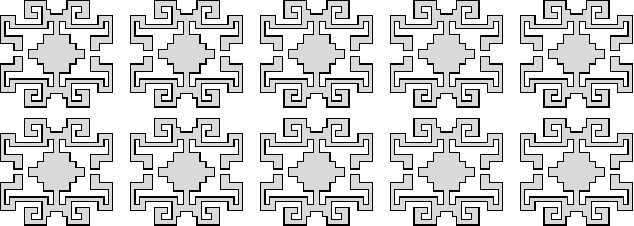}
\caption{The Delecroix--Zorich variant (\cf \cite[Figure 4]{Delecroix:Zorich}).}
\label{fig:DZ-wtm}
\end{figure}
\end{Remark}
The general discussion about diffusion rates above remains valid in this case (see \Cref{r:DZ-decomposition}) and we have the following.
\begin{Theorem}
\label{thm:main-DZ}
For any Delecroix--Zorich wind-tree model with {\em rational side lengths} one has
$$
[0,1) \subset \left\{\delta_\theta;\, \theta\in \mathcal D \right\} \subset [0,1].
$$
\end{Theorem}
We dont know whether $1$ belongs to $\left\{\delta_\theta;\, \theta\in \mathcal D \right\}$ in this case.

\subsection*{Joint diffusion}
We also exhibit an infinite family of wind-tree billiards
with all possible joint diffusion rates.
To our knowledge, this is a phenomenon that has not been previously exhibited.
More precisely, consider the \emph{horizontal diffusion rate}
$$
\delta^{\mathrm{h}}_\theta(x) = \limsup_{T \to +\infty} \frac{\log d_{\mathrm{h}}(x,\phi^\theta_T(x))}{\log T} \in [0,1],
$$
where $d_{\mathrm{h}}(.,.)$ is the Euclidean horizontal distance on $\RR^2$, that is, $d_{\mathrm{h}}((a,b),(c,d)) = |c - a|$.
Similarly, we define the \emph{vertical diffusion rate} $\delta^{\mathrm{v}}_\theta(x)$.
As in the case of the diffusion rate $\delta_\theta(x)$, the horizontal and vertical diffusion rates are constants for $\theta$ in a full measure set $\mathcal D' \subset \mathcal D \subset \mathbb S^1$, and denoted $\delta^{\mathrm{h}}_\theta, \delta^{\mathrm{v}}_\theta$, respectively.
We prove the following.

\begin{Theorem}
\label{thm:square}
For side lengths $(a,b)$ in the set
\[\esq = \left\{ \left(\frac{p}{q},\frac{r}{s}\right) \in (0,1)^2;\, \gcd(p,q) = \gcd(r,s) = 1, \ p,q,r,s \in 2\NN-1\right\},\]
the set of joint diffusion rates contains the full open square. More precisely,
$$
\{(0,0)\}\cup(0,1)^2 \subset \{(\delta^{\mathrm{h}}_\theta, \delta^{\mathrm{v}}_\theta);\, \theta \in \mathcal D'\}
\subset [0,1]^2.
$$
\end{Theorem}

\subsection*{Lyapunov spectrum of hyperbolic surfaces}
\label{sec:la:spec}

It is now well established (see the approach originated in the pioneering work~\cite{Delecroix:Hubert:Lelievre})
that the diffusion rate can be interpreted as a Lyapunov exponent of a certain renormalizing
dynamical system associated to the billiard flow.
In the case of $(a,b) \in \mathcal E$, this corresponds to the geodesic flow on a hyperbolic surface
(see \Cref{sec:le} for more details).

The results presented in this section are independent to the wind-tree model and will serve for our purposes.

By $\mathbb H$ we denote the hyperbolic plane, identified with $\PSO(2,\RR)\backslash \PSL(2,\RR)$.
Its unit tangent space may be identified with $\PSL(2,\RR)$.
Let $S=\mathbb H/\Gamma$ be a hyperbolic surface,
where $\Gamma$ is a (non necessarily uniform) lattice of $\PSL(2,\RR)$.
The geodesic flow $(g_t)$ coincides with the (left) action of the diagonal group
$\left\{\left(\begin{smallmatrix} e^{t/2} & 0\\ 0 & e^{-t/2} \end{smallmatrix}\right);\, t \in \RR\right\}$
on $X= \PSL(2,\RR)/\Gamma$.

We consider a $2$-dimensional symplectic bundle $p\colon E\to X$ and a cocycle $A$ over $g_t$, \ie, a family of symplectic linear maps $A^t(x)\colon E(x)\to E(g_t(x))$, which depends continuously on $(x,t)\in X\times \RR$
and satisfies $A^s(g_t(x))\circ A^t(x)=A^{s+t}(x)$.

We assume that the cocycle is obtained as the \emph{suspension}
of a representation
$$
\rho\colon \Gamma\to \SL(2,\RR).
$$
That is, we consider the trivial bundle $\widehat E=\PSL(2,\RR)\times \RR^2$
and the trivial cocycle defined by $\widehat A^t(x).u=(g_t(x),u)$. The
group $\Gamma$ acts on $\widehat E$ by $\gamma(x,u)=(\gamma(x),\rho(\gamma).u)$. The bundle $E$
is then the quotient $\widehat E/\Gamma$. Since the actions of $(g_t)$ and $\Gamma$ commute on $\SL(2,\RR)$,
the cocycle $\widehat A$ on $\widehat E$ induces a cocycle $A$ on $E$.

We introduce the set $\mathcal M$ of Borel probability measures $\mu$ on $X$ that are invariant, ergodic for $g_t$
and \emph{Oseledets regular}, \ie, there exists $\lambda(\mu)\geq 0$ such that either
\begin{itemize}
\item $\lambda(\mu)=0$ and $\frac 1 t \log \|A^t(x)\|\to 0$ as $t\to \pm \infty$, for $\mu$-a.e. $x\in X$, or
\item $\lambda(\mu)>0$ and there is a measurable splitting $E=E^s\oplus E^u$
such that $\frac 1 t \log \|A^t(x)|_{E^s}\|\to -\lambda(\mu)$ and $\frac 1 t \log \|A^t(x)|_{E^u}\|\to \lambda(\mu)$
as $t\to \pm \infty$, for $\mu$-a.e. $x\in X$. Moreover $A^t(x).E^{s/u}(x)=E^{s/u}(g_t(x))$
for $\mu$-a.e. $x\in X$ and all $t\in \RR$.
\end{itemize}
The number $\lambda(\mu)$ is then uniquely defined and it will be referred to as the \emph{Lyapunov exponent} of $\mu$.

By Oseledets theorem, for a given ergodic probability measure $\mu$, if
the maps $x\mapsto |\log \|A^t(x)\||$ are integrable for each each $t\in \RR$, then $\mu\in\mathcal M$.
This is the case in particular to the compactly supported measures.
We will prove the following.

\begin{Theorem}\label{t.main}
The set $\Lambda_1\coloneqq\{\lambda(\mu);\, \mu \in \mathcal M\}$ is an interval.
\end{Theorem}
This may be generalized as follows. We consider:
\begin{itemize}
\item two cocycles $A,A'$ on $E$ over $(g_t)$, defined as suspensions of representations $\rho,\rho'$,
\item the set of regular measures $\mathcal{M}\coloneqq\mathcal M_{A}\cap \mathcal M_{A'}$,
\item for each $\mu\in \mathcal{M}$, the Lyapunov vector
$(\lambda(\mu),\lambda'(\mu))\coloneqq(\lambda_{A}(\mu),\lambda_{A'}(\mu))$.
\end{itemize}

\begin{Theorem}\label{t.main2}
Let $\Lambda\coloneqq\{(\lambda(\mu),\lambda'(\mu));\, \mu\in \mathcal{M}\}$
and let $L\subset \mathbb{R}^2$ be the smallest affine space containing $\Lambda$.
If $\rho$, $\rho'$ are irreducible, the interior of $\Lambda$ (relative to $L$) is convex and dense in~$\Lambda$.
\end{Theorem}

\begin{Remark} Theorems~\ref{t.main} and~\ref{t.main2} are closely related to the work of Breuillard and Sert (see \cite{Breuillard:Sert}) for random products of matrices. A random version of Theorems \ref{t.main} is proven in Feng (see \cite{Feng})  in a different setting. 
\end{Remark}

To each wind tree model are associated some lattice $\Gamma$, two irreducible representations $\rho^-,\rho^+$
and, by the previous discussion, some set $\Lambda$ which satisfies
$\Lambda\subset \{(\delta^{\mathrm{h}}_\theta, \delta^{\mathrm{v}}_\theta);\, \theta \in \mathcal D'\}$.
A main part of this work is to control the set $\Lambda$ by showing that
$(0,0)\in \Lambda$ and $(1,1)\in \overline \Lambda$; in the setting of \Cref{thm:square}
we also prove that $\Lambda$ contains the full open square $(0,1)^2$.

\begin{Remark}
For the wind-tree model associated to the surface with three squares
that is, when the scatterers are squares of side length one half, $\Lambda$ consists of the segment $x=y$:
indeed the representations $\rho^-,\rho^+$ are conjugated, see~\cite[Section~2.3.2]{Pardo:quantitative}. This contrasts with \Cref{thm:square}. It would be interesting to have instances of $\Lambda$ with different (intermediate) shapes.
Also that we don't know how much can differ the sets $\Lambda$ and $\{(\delta^{\mathrm{h}}_\theta, \delta^{\mathrm{v}}_\theta);\, \theta \in \mathcal D'\}$.
\end{Remark}

\subsection*{Structure of the paper}

In \Cref{sec:connectedness} we prove \Cref{t.main} and \Cref{t.main2} which are needed for the proof of \Cref{thm:main} and \Cref{thm:square}, respectively.

In \Cref{sec:background} we review useful results concerning translation surfaces, their moduli spaces, affine invariant submanifolds and the Kontsevich--Zorich cocycle.

In \Cref{sec:le} we relate the diffusion rate of the wind-tree model to the Lyapunov exponents of the Kontsevich--Zorich cocycle following \cite{Delecroix:Hubert:Lelievre,Delecroix:Zorich}.

In \Cref{sec:parabolic} we give a geometric criterion to classify whether a parabolic element is in the kernel of the natural representations.

In \Cref{sec:endpoints} we conclude the proof of \Cref{thm:main} by showing that $(0,0) \in \Lambda$ and $(1,1) \in \overline\Lambda$.

In \Cref{sec:square} we prove \Cref{thm:square} exhibiting an infinite family of wind-tree billiards
for which the Lyapunov spectrum contains the full square $(0,1)^2$.

\subsection*{Acknowledgments} The authors thank Yves Benoist and Igor Krichever for helpful conversations.
This work was partially supported by the ANR Project GeoDyM, the LabEx PERSYVAL-Lab (ANR-11-LABX-0025-01) and the ERC project 692925 \emph{NUHGD}.
This work was also partially supported by Centro de Modelamiento Matemático (CMM), ACE210010 and FB210005, BASAL funds for centers of excellence from ANID-Chile, and the MATH-AmSud 21-MATH-07 grant.
The fourth named author was also supported by ANID-Chile through the FONDECYT 3190257 and 1221934 grants.

\section{Lyapunov spectrum}
\label{sec:connectedness}

In this section we prove \Cref{t.main} and \Cref{t.main2} which are used latter for the proof of \Cref{thm:main} and \Cref{thm:square}, respectively.

The proof of \Cref{t.main2} has two cases:
\begin{description}
\item[\hypertarget{case1}{Case 1}] $\Lambda$ is contained in a line.
It is then enough to consider one cocycle and we are reduced to \Cref{t.main}.
\item[\hypertarget{case2}{Case 2}] $\Lambda$ is not contained in a line.
\end{description}
The proof in the \hyperlink{case1}{first case} (or the proof of \Cref{t.main}) is very similar
to the proof of the \hyperlink{case2}{second case}, but simpler.
In the following we only detail the \hyperlink{case2}{second case}.

\subsection{Periodic approximation}
We begin by establishing a property of Lyapunov exponents supported on periodic orbits (\cf~\cite{Benoist}).

\begin{Proposition}\label{p.per}
$\Lambda_{per}=\{(\lambda(\mu),\lambda'(\mu)),\;\mu \in \mathcal M
\text{ supported on a periodic orbit}\}$ is dense in~$\Lambda$.
\end{Proposition}

\begin{proof} Let us consider an arbitrary measure $\mu\in \mathcal M$.
Any point $x$ in a set with full $\mu$-measure is recurrent and satisfies Oseledets theorem.
One lifts $x$ as a point $\widehat x\in \PSL(2,\RR)$ and chooses a small neighborhood $U$ disjoint from
its images $\gamma(U)$, $\gamma\in \Gamma\setminus \{e\}$.
Since $x$ is recurrent by the flow,
there exists $t>0$ large and $\gamma\in \Gamma$
such that $\gamma(g_t(\widehat x))$ is arbitrarily close to $\widehat x$.
The Anosov closing lemma then gives $\widehat z\in U$ and $s>0$ such that $\gamma(g_s(\widehat z))=\widehat z$.
Moreover $|s-t|$ is arbitrarily small if $\gamma(g_t(\widehat x))$ is close enough to $\widehat x$.
We denote by $z$ the projection of $\widehat z$ to $X$: it is fixed by $g_s$.
Since the bundle $E$ is locally trivial it is defined from a representation $\rho\colon \Gamma\to \SL(2,\RR)$,
so that taking a chart over $U$, one has $A^t(x)=A^s(z)=\rho(\gamma)$.
\bigskip

Let us first assume that $\lambda(\mu)=0$.
Having chosen $t$ large, $\frac 1 t \log \|A^t(x)\| $ is close to $0$.
Since $|s-t|$ is small, one deduces that $\frac 1 s \log \|A^s(z)\|$ is also close to zero.
We have built a periodic orbit whose Lyapunov exponent associated to the cocycle $A$ is arbitrarily close to $0$.
\bigskip

We then assume $\lambda(\mu)>0$: there exists a measurable hyperbolic decomposition $E= E^s\oplus E^u$.
We may have chosen $x$ and $g_t(x)$ in a
set $Y\subset X$ with positive $\mu$-measure where the bundles $E^s,E^u$
vary continuously and are invariant by $A^t(x)$.
We then check that the Lyapunov exponent $\lambda(z)$ of the periodic orbit of $z$
is close to $\lambda(\mu)$.
(This is similar to \cite[Theorem 3.8]{Abdenur-Bonatti-Crovisier}.)

Any vector $v$ in $\RR^2$ decomposes as $v=v_x^s+v_x^u$ according to the splitting
$E^s(x)\oplus E^u(x)$.
Let us choose $\alpha>0$ small and consider the cone $\mathcal C$ of vectors $v$
satisfying $\|v_x^s\|\leq \alpha \|v_x^u\|$.
For $y\in Y$ close to $x$, the decomposition $E^s(y)\oplus E^u(y)$ is close,
hence any vector $v=v_y^s+v_y^u\in E^s(y)\oplus E^u(y)$
may be decomposed as $v_x^s+v_x^u\in E^s(x)\oplus E^u(x)$
such that $\|v_x^s\|\leq \alpha^2 \|v_y^u\|+(1+\alpha^2)\|v_y^s\|$
and $\|v_x^u\|\geq (1-\alpha^2) \|v_y^u\|-\alpha^2\|v_y^s\|$.

For $t$ large enough, $\|A^t(x).v_x^s\|\leq \frac 1 2 \|v_x^s\|$ and $\|A^t(x).v_x^u\|\geq 2 \|v_x^u\|$.
Since $g_t(x)\in Y$ is close to $x$,
the image
$\bar v\coloneqq A^t(x).v$ decomposes as $\bar v^s+\bar v^u\in E^s(x)\oplus E^u(x)$
satisfying
$$\|\bar v^s\|\leq \big(\alpha^2+(1+\alpha^2)\tfrac \alpha4\big)\; \|A^t(x).v^u_x\|
\text{ and }
\|\bar v^u\|\geq \big(1-\alpha^2-\tfrac{\alpha^3} 4\big)\; \|A^t(x).v^u_x\|.$$
Hence, having chosen $\alpha$ small enough, the image $\bar v$ is still in the cone $\mathcal C$.
Moreover
$$\|A^t(x)\|\geq \frac{\|\bar v\|}{\|v\|}\geq \big(1-\alpha^2-\tfrac {\alpha^3} 4\big)\;\frac{1-\alpha}{1+\alpha}\; \frac{\| A^t(x).v^u_x\|}{\|v^u_x\|}.$$
Since both $\frac 1 t \log \|A^t(x)\| $ and $\frac 1 t \log \|A^t(x)|_{E^u(x)}\| $ are close to $\lambda(\mu)$, and since $|s-t|$ is small,
this proves that vectors in the cone $\mathcal C$ grow exponentially under forward iterations of $A^s(z)$
with a rate close to $\lambda(\mu)$.
The Lyapunov exponent of the invariant measure supported on the orbit of $z$ is thus close to $\lambda(\mu)$ as required.

The same analysis applies to the other cocycle $A'$. Hence the two Lyapunov exponents
$\lambda(z),\lambda'(z)$ of the periodic orbit of $z$ are close to those of $\mu$.
\end{proof}

\subsection{Reduction to a symbolic setting}

We will work with the alphabet $\{1, 2, 3\}$.
Let $\Sigma\coloneqq\{1,2,3\}^{\mathbb{Z}}$ be the space of sequences $\underline \varepsilon=(\varepsilon_k)_{k\in \ZZ}$ endowed with the shift map $\sigma$. If $w=\varepsilon_0\varepsilon_1\cdots \varepsilon_{n-1}$ is a finite word, we denote by
$\overline{w}$ the periodic sequence $\underline \varepsilon$ such that
$\varepsilon_{in+j}=\varepsilon_j$ for each $i\in \ZZ$ and $0\leq j<n$.
Also $i^\ell$ denotes the word $i\cdots i$ where the symbol $i$ is repeated $\ell$ times.
And $[i]\coloneqq\{\underline \varepsilon,\; \varepsilon_0=i\}$ is the $1$-cylinder associated to the symbol $i$.

For any continuous function $r\colon \Sigma \to (0,\infty)$, one defines
the \emph{suspension} $\Sigma_r$ as the quotient of $\Sigma\times \RR$ by the map
$(\underline \varepsilon,s)\mapsto (\sigma(\underline \varepsilon), s- r(\underline \varepsilon))$.
It is endowed with the flow $(\sigma_r^t)_{t\in \RR}$ induced by
$\sigma_r^t(\underline \varepsilon, s)= (\underline \varepsilon, s+t)$.
\medskip

A \emph{cocycle} $M$ on $\Sigma$ is a map which associates
to $(\underline\varepsilon,n)\in\Sigma\times \ZZ$ a matrix $M^n(\underline\varepsilon)\in \SL(2,\RR)$
and satisfies
$M^m(\sigma^n(\underline \varepsilon))\circ M^n(\underline \varepsilon)=M^{n+m}(\underline \varepsilon)$.
It defines by suspension a cocycle $\mathcal{A}$ on $\Sigma_r$: let $\mathcal{E}\to \Sigma_r$ be the bundle which is the quotient
of $(\Sigma\times \RR)\times \RR^2$ by the map
$$(\underline \varepsilon,s,v)\mapsto (\sigma(\underline \varepsilon), s- r(\underline \varepsilon),M^1(\underline \varepsilon).v).$$
The cocycle $\mathcal{A}$  acts on $\mathcal{E}$ and is the quotient of the trivial cocycle
$\mathcal{A}^t(\underline \varepsilon,s,v)=(\underline \varepsilon,s+t,v)$.

A cocycle on $\Sigma$ is \emph{one-step locally constant} if $M^1(\underline \varepsilon)$
only depends on $(\varepsilon_0,\varepsilon_1)$. Such a cocycle is defined by $9$ linear maps
$M_{i,j}\in \SL(2,\RR)$ with $i,j\in\{1,2,3\}$.
For any word $w=\varepsilon_0\cdots\varepsilon_{n}$, we denote
$$M_w=M_{\varepsilon_0\cdots\varepsilon_{n}}=M_{\varepsilon_{n-1},\varepsilon_{n}}\cdots M_{\varepsilon_0,\varepsilon_1},$$
\ie, $M_w$ coincides with $M^n(\underline \varepsilon')$ associated to the sequences $\underline \varepsilon'$ such that $\varepsilon'_i=\varepsilon_i$ for $0\leq i\leq n$.

A one-step locally constant cocycle $M$ on $\Sigma$ is \emph{irreducible} if there does not exist
a linear $1$-space $L\subset \RR^2$ which is invariant by each map $M_w$ where $w=\varepsilon_0\cdots\varepsilon_{n}$
satisfies $\varepsilon_0=\varepsilon_n$.
\medskip

One can extract a symbolic system from our geometrical setting:

\begin{Lemma}
\label{lm:1}
For any periodic points $z_1,z_2,z_3\in X$ with disjoint orbits, there exist:
\begin{itemize}
\item a invariant compact set $K\subset X$ containing $z_1,z_2,z_3$ and invariant by the flow $g$,
\item a continuous function $r\colon \Sigma=\{1,2,3\}^\ZZ \to (0,\infty)$,
\item a conjugacy $h\colon \Sigma_r\to K$ between $(\Sigma_r,\sigma_r)$ and $(K,g)$
such that $h(\overline i,0)=z_i$,
\item one-step locally constant cocycles $M,M'$,
defining cocycles $(\mathcal{E},\mathcal{A})$,  $(\mathcal{E}',\mathcal{A}')$ on $(\Sigma_r,\sigma_r)$,
that are conjugated to the the cocycles $(E, A)$, $(E, A')$, on $(K,g)$,
\end{itemize}
Moreover, if the representations $\rho,\rho'$ are irreducible and $\lambda(z_1),\lambda'(z_1)\neq 0$, then the previous is also true with $M,M'$ irreducible cocycles.
\end{Lemma}
\begin{proof}
The shadowing lemma provides us with an invariant hyperbolic set $K$ contained in a small neighborhood $U$ of
the orbits of the periodic points $z_i$ and of heteroclinic points $z_{j,k}$.
Fixing pieces of orbits $\zeta_{j,k}$ connecting a neighborhood of $z_j$ to a neighborhood of $z_k$,
the set $K$ is the union of orbits that can be split as pieces of orbits close to
the orbits of $z_i$ and transition arcs $\zeta_{j,k}$. This decomposition defines a coding over the alphabet
$\{1,2,3\}$, implying the three first items of the lemma.

By pulling back by $h$ the cocycles $(E,A)$ and $(E,A')$, one defines cocycles $(\mathcal{E},\mathcal{A})$,  $(\mathcal{E}',\mathcal{A}')$ on $(\Sigma_r,\sigma_r)$.
Since $(E,A)$ and $(E,A')$ are the suspensions of representations $\rho,\rho'\colon \Gamma\to \SL(2,\RR)$,
the cocycles $(\mathcal E,\mathcal A)$, $(\mathcal{E}',\mathcal{A}')$ are the suspensions of locally constant cocycles $M,M'$.
Having chosen the neighborhood $U$ small (hence the pieces of orbits $\zeta_{j,k}$ long)
the sets $h([i],0)$ have small diameters and the cocycles $M,M'$ are one-step locally-constant, giving the last item.
\medskip

Let us assume that $\rho,\rho'$ are irreducible and $\lambda(z_1),\lambda'(z_1)\neq 0$,
and let us prove
the final statement.
We claim that one can modify the heteroclinic points $z_{1,2}$
in such a way that there does not exist a linear $1$-space $L\subset \RR^2$
preserved by the matrices $M_{1,1}$ and
$M_{2,1}M_{1,2}$. This will imply that $M$ is irreducible.
In a similar way, one can modify independently $z_{1,3}$
in such a way that there is no $1$-space preserved by $M'_{1,1}$ and
$M'_{3,1}M'_{1,3}$. This will imply that $M'$ is irreducible.

Since $\lambda(z_1)\neq 0$, the matrix $M_{1,1}$ is hyperbolic and preserves two directions
$E^u,E^s$. One will select two elements $\gamma_u,\gamma_s\in \Gamma$
and modify the heteroclinic orbit $\zeta_{1,2}$ by an orbit that shadows $\ell$ times $\gamma_u$,
then $m$ times the orbit of $z_1$, 
then $n$ times $\gamma_s$, and finally the original orbit $\zeta_{1,2}$.
In this ways, one modifies the matrix $M_{1,2}$ by precomposing by
a product $A^\ell M_{1,1}^m B^n$. Since $\rho$ is irreducible, one can choose $A$ that does not preserves $E^u$
and $B$ that does not preserves $E^s$.

Up to replace $A$ by $A^2$, the matrix $M_{2,1}M_{1,2}A$ does not preserve $E^u$.
If it does not preserve $E^s$ either, one replaces $M_{1,2}$ by $M_{1,2}A$ as explained and gets the required property.

If $M_{2,1}M_{1,2}A E^s=E^s$,
the matrices $M_{2,1}M_{1,2}AM_{1,1}^\ell B$ do not preserve $E^s$.
Since $M_{2,1}M_{1,2}A$ does not preserve $E^u$, the directions $M_{2,1}M_{1,2}A E^u$ and $M_{2,1}M_{1,2}A M_{1,1} E^u$
are different. Hence one can choose $\ell=0$ or $1$, and ensure that
$M_{2,1}M_{1,2}A M_{1,1}^\ell B$ does not preserve $E^u$.
In this case, one thus replaces $M_{1,2}$ by $M_{1,2}A B$ or $M_{1,2}AM_{1,1}B$ and get the property.

This concludes the proof of the irreducibility of the cocycles $M,M'$.
\end{proof}

In the next sections, we construct new suspended shifts $\Sigma_{\widetilde r}$
with cocycles $\widetilde {\mathcal{A}},\widetilde {\mathcal{A}'}$ that are isomorphic
to the cocycles $\mathcal A,\mathcal A'$ over invariant compact sets of $\Sigma_r$.
They are obtained by selecting periodic orbits and heteroclinic points in $\Sigma_r$,
in the same way as for proving \Cref{lm:1}.

\begin{Lemma}
\label{lm:1.b}
For any periodic points $z_1,z_2,z_3\in \Sigma_r$ with disjoint orbits, there exist:
\begin{itemize}
\item a continuous function $\widetilde r\colon \Sigma=\{1,2,3\}^\ZZ \to (0,\infty)$,
\item a conjugacy $h\colon \Sigma_{\widetilde r}\hookrightarrow \Sigma_r$ between $\Sigma_{\widetilde r}$ and
a subset $K$ of $\Sigma_r$
such that $h(\overline i,0)=z_i$,
\item one-step locally constant cocycles $\widetilde M,\widetilde M'$,
defining cocycles $(\widetilde{\mathcal{E}},\widetilde{\mathcal{A}})$,  $(\widetilde{\mathcal{E}}',\widetilde{\mathcal{A}}')$ on $(\Sigma_{\widetilde r},\sigma_{\widetilde r})$,
that are conjugated to the cocycles $(\mathcal{E},\mathcal{A})$,
$(\mathcal{E}',\mathcal{A}')$ on $(K,\sigma_{\widetilde r})$,
\end{itemize}
\end{Lemma}

\subsection{Approximation by hyperbolic periodic orbits}
One can upgrade \Cref{p.per} and require that the periodic orbits are hyperbolic.

\begin{Proposition}\label{p.hyperbolic-periodic}
Let $M,M'$ be one-step locally constant cocycles on $\Sigma$ such that $\lambda(\overline 1)\neq 0$.
For any $i\in\{1,2,3\}$, there exists a periodic orbit in $\Sigma_r$ such that the Lyapunov
exponents for $\mathcal{A},\mathcal{A}'$ are arbitrarily close to the exponents $\lambda(\overline i),\lambda'(\overline i)$
and the exponent for $\mathcal{A}$ is non-vanishing.
\end{Proposition}

We first prove the property for a single cocycle.
\begin{Lemma}\label{l.hyperbolic-periodic}
Let $M$ be a one-step locally constant cocycle on $\Sigma$ such that $\lambda(\overline 1)\neq 0$.
For any $i\in\{1,2,3\}$, there exists a periodic orbit in $\Sigma_r$ whose Lyapunov
exponent for $\mathcal{A}$ is arbitrarily close to the exponent $\lambda(\overline i)$
and is non-vanishing.
\end{Lemma}
\begin{proof} Without loss of generality, one may consider the case $i=2$ and assume that the spectral radius of
$M_{2,2}$ vanishes ($\lambda(\overline 2)=0$.) We will choose $\ell\geq 2$ and take $n$ large, so that the exponent of the periodic orbit $\gamma$ in $\Sigma_r$
which follows the itinerary $1^\ell2^n$ in $\Sigma$ is arbitrarily close to $0$.

Let us first assume that $M_{2,2}$ is conjugated to a rotation.
The integer $\ell$ is chosen arbitrarily.
Let $\mathcal{C}^u$ be a small open neighborhood of the unstable space $E^u$ of $M_{1,1}$
such that $M_{1,1}\overline{\mathcal{C}^u}\subset \mathcal{C}^u$.
Since the rotations are recurrent,
there exists $n$ large such that $M_{2,2}^nM_{1,1}^\ell\overline{\mathcal{C}^u}\subset \mathcal{C}^u$.
This implies that the matrix $M_{2,2}^nM_{1,1}^\ell$ is hyperbolic.
Hence the Lyapunov exponent for $\mathcal A$ of the periodic orbit $\gamma$ is non-vanishing.

Let us now suppose that $M_{2,2}$ is parabolic (and not the identity).
It admits a unique invariant $1$-space $E^c$.
If $E^c$ coincides with one of the invariant spaces $E^s$ or $E^u$ of $M_{1,1}$,
the spectral radius of $M_{2,2}^nM_{1,1}^\ell$ coincides with the spectral radius of $M_{1,1}^\ell$,
hence the Lyapunov exponent for $\mathcal A$ of the periodic orbit $\gamma$ is non-vanishing.

If $M_{2,2}$ is parabolic and $E^c$ does not coincides with $E^s$ or $E^u$,
one fixes an open set $\mathcal C\subset P^1(\RR)$ which is the complement of a small neighborhood of
$E^s$ and chooses $\ell\geq 1$ large so that
$M_{1,1}^\ell \mathcal C$ is a small neighborhood of $E^u$.
Taking $n$ large enough, $M_{2,2}^nM_{1,1}^\ell \overline{\mathcal C}$ is close to $E^c$, hence contained in $\mathcal C$.
This implies that the matrix $M_{2,2}^nM_{1,1}^\ell$ is hyperbolic and
that the Lyapunov exponent for $\mathcal A$ of the periodic orbit $\gamma$ is non-vanishing.
\end{proof}

\begin{proof}[Proof of \Cref{p.hyperbolic-periodic}]
We continue the proof of \Cref{l.hyperbolic-periodic}, still assuming $i=2$ and $\lambda(\overline 2)=0$:
we get a periodic orbit $\gamma$ which follows the itinerary $1^\ell2^n$ with $n$ large.
As before the Lyapunov exponent for $\mathcal A$ is close to $\lambda(\overline 2)$ and is non vanishing,
so it remains to control the Lyapunov exponent for $\mathcal A'$.

Let us first consider the case $\lambda'(\overline 2)=0$. Since $n$ is large, the Lyapunov exponent of $\gamma$ for $\mathcal A'$ is close to $\lambda'(\overline 2)=0$; so the proposition holds.

In the case $\lambda'(\overline 2)\neq 0$,
let ${E^s}'$ and ${E^u}'$ be the stable and unstable spaces of $M'_{2,2}$.
Regarding the itinerary $1^\ell2^n$,
note that the integer $\ell$ may be chosen so that $(M'_{1,1})^\ell$ does not send ${E^u}'$ on ${E^s}'$.
Let $\mathcal{C}'$ be a small open neighborhood of ${E^u}'$, whose closure is disjoint from ${E^s}'$.
Provided that $n$ is larger than some $n_0\geq 1$, we thus have
$(M'_{2,2})^n(M'_{1,1})^\ell\overline{\mathcal{C}'}\subset \mathcal{C}'$. The matrix $(M'_{2,2})^n(M'_{1,1})^\ell$ is hyperbolic
and the difference between its spectral radius differs and the spectral radius of $(M'_{2,2})^n$ is bounded uniformly in $n$.
This implies that the Lyapunov exponent of $\gamma$ for $\mathcal A'$ is close to $\lambda'(\overline 2)$ and
the proposition holds in this case also.
\end{proof}

\subsection{Hyperbolic cocycles}

A cocycle $M$ on $\Sigma$
is \emph{uniformly hyperbolic} if  for each $\underline \varepsilon\in\Sigma$, there exists a decomposition
$\RR^2=E^s(\underline \varepsilon)\oplus E^u(\underline \varepsilon)$
which
\begin{itemize}
\item depends continuously on $\underline \varepsilon$,
\item is invariant, \ie, $M^1(\underline \varepsilon).E^{s}(\underline \varepsilon)=E^{s}(\sigma(\underline \varepsilon))$
and $M^1(\underline \varepsilon).E^{u}(\underline \varepsilon)=E^{u}(\sigma(\underline \varepsilon))$,
\item is contracted/expanded, \ie, $\|M^n|_{E^s}\|\leq 1/2$ and $\|M^{-n}|_{E^u}\|\leq 1/2$ for some $n\geq 1$.
\end{itemize}

The hyperbolicity can be checked with the following \emph{cone criterion}.

\begin{Proposition}[{\cite[Theorem~2.3]{Avila-Bochi-Yoccoz}}]\label{p.ABY}
A one-step locally constant cocycle $M$ on $\SL(2,\RR)$ is uniformly hyperbolic
if and only if there exist non-empty open sets $\mathcal C_1,\mathcal C_2,\mathcal C_3\subsetneq P^1(\RR)$
satisfying $M_{i,j}(\overline{\mathcal C_i})\subset \mathcal C_j$ for any $i,j\in \{1,2,3\}$.
\end{Proposition}

We then explain how to extract a uniformly hyperbolic subsystem.

\begin{Lemma}\label{l.extract-hyperbolic}
Let $\Sigma_r$ be a suspended shift and
let $M,M'$ be irreducible one-step locally constant cocycles on $\Sigma$
such that the $3$ fixed points $\overline i\in \Sigma$ are hyperbolic for each cocycle $M,M'$.

Then, there exists an extracted suspended shift $h\colon \Sigma_{\widetilde r}\hookrightarrow \Sigma_r$
which satisfies $h(\overline i,0)=(\overline i,0)$ for each $i\in\{1,2,3\}$ and two uniformly hyperbolic one-step locally constant
cocycles $\widetilde M, \widetilde M'$ on $\Sigma$ whose suspensions on $\Sigma_{\widetilde r}$
are conjugated to the suspensions on $M,M'$ over $h(\Sigma_{\widetilde r})$.
\end{Lemma}
\begin{proof}
Let $\RR^2=E^s_i\oplus E^u_i$ be the hyperbolic decomposition over the fixed point $\overline i$.
\begin{NoNumberClaim} For every $i\neq j$, there exists a finite word $w(i,j)$ starting with $i$ and ending with $j$
such that $M_{w(i,j)} E^s_i\neq E^u_j$ and $M_{w(i,j)} E^u_i\neq E^s_j$.
\end{NoNumberClaim}
\begin{proof}[Proof of the claim]
One may assume that $M_{i,j} E^s_i= E^u_j$ or $M_{i,j} E^u_i= E^s_j$, otherwise we are done.
We will
consider the first case (the second one is analogous).
Since $M$ is irreducible, there is a word $w$ which starts and ends with $i$ satisfying
$M_{w}E^s_i\neq E^s_i$. Note also that $M^2_{w}E^s_i\neq E^s_i$
and $M^2_{w}E^s_i\neq M_{w}E^s_i$. In particular we have $M_{i,j}M_w E^s_i\neq E^u_j$ and
$M_{i,j}M_w^2 E^s_i\neq E^u_j$.

If $M_{i,j}M_wE^u_i\neq E^s_j$ or $M_{i,j}M_w^2E^u_i\neq E^s_j$ we are done with $w_{i,j}=wj$ or $w_{i,j}=wwj$.
Otherwise $M_wE^u_i=E^u_i$ and $M_{i,j} E^u_i=E^s_j$.
We choose another word $w'$ which starts and ends with $i$ satisfying $M_{w'}E^u_i\neq E^u_i$.
Hence we have $M_{i,j}M_{w'}E^u_i\neq E^s_i$ and $M_{i,j}M_{w'w}E^u_i\neq E^s_i$.

If $M_{i,j}M_{w'} E^s_i\neq E^u_j$ or $M_{i,j}M_{w'w} E^s_i\neq E^u_j$, we are done
with $w_{i,j}=w'j$ or $w_{i,j}=w'w'j$.
Otherwise $M_{w'}E^s_i=E^s_i$ so that
$M_{i,j}M_{w'w} E^s_i\neq E^u_j$ and $M_{i,j}M_{w'w'w} E^s_i\neq E^u_j$.
Since $M_{w'}E^u_i\neq E^u_i$, the images of $E^u_i$ by $M_{i,j}M_{w'w}$
and $M_{i,j}M_{w'w'w}$ are different, so we are done
with the word $w_{i,j}= w'wj$ or $w_{i,j}= w'w'wj$.
\end{proof}

Let $\mathcal{C}_1,\mathcal{C}_2,\mathcal{C}_3\subset P^1(\RR)$ be small open neighborhoods of
$E^u_1,E^u_2,E^u_3$ such that  $M_{i,i} \overline{\mathcal{C}_i}\subset \mathcal{C}_i$
and $E^s_j\not\in M_{w(i,j)} \overline{\mathcal{C}_i}$.
Choosing any $\ell\geq 1$ large, $M_{j,j}^\ell$ contracts $M_{w(i,j)} \overline{\mathcal{C}_i}$
near the unstable direction $E^u_j$, so that
$M_{j,j}^\ell M_{w(i,j)} \overline{\mathcal{C}_i}\subset \mathcal{C}_j$.
The same construction applies to the cocycle $M'$.

One extracts a suspended subshift $h\colon \Sigma_{\widetilde r}\to \Sigma_r$
with the same periodic points $h(\overline i,0)=(\overline i, 0)$, but whose
transition between different symbols $i,j$ follows the itinerary $w(i,j)\;j^\ell$ from the original coding.
By \Cref{lm:1.b}, one gets new cocycles $(\widetilde{\mathcal{E}},\widetilde{\mathcal{A}})$,  $(\widetilde{\mathcal{E}}',\widetilde{\mathcal{A}}')$ over $\Sigma_{\widetilde r}$ which are associated to
one-step locally constant cocycles $\widetilde M,\widetilde M'$ over $\Sigma$.
By construction $\widetilde M_{i,i}=M_{i,i}$, whereas $\widetilde M_{i,j}=M_{j,j}^\ell M_{w(i,j)}$.
Consequently the cone criterion in \Cref{p.ABY} holds for $\widetilde M$
and the sets $\mathcal{C}_i$, proving that $\widetilde M$ is uniformly hyperbolic. The same holds for $\widetilde M'$.
\end{proof}

\subsection{Proof of~\Cref{t.main2}}
As we explained in at the beginning of \Cref{sec:connectedness} we will consider the main case, where the set $\Lambda\coloneqq\{(\lambda(\mu),\lambda'(\mu)), \mu\in \mathcal{M}\}$ is not contained in a line.

\begin{Proposition}\label{p.per2}
$\Lambda^h_{per}=\{(\lambda(\gamma),\lambda'(\gamma)),\; \gamma \text{ periodic orbit}
\text{ with } \lambda(\gamma),\lambda'(\gamma)$$\neq 0\}$ is dense in~$\Lambda$.
\end{Proposition}
\begin{proof}
By \Cref{p.per}, it is enough to prove that $\Lambda^h_{per}$ is dense in $\Lambda_{per}$.
Moreover there exists three periodic orbits $\gamma_i$ in $X$ whose
pairs of exponents $(\lambda(\gamma_i),\lambda'(\gamma_i))$ are not contained in a common line.
This implies that there exists  two periodic orbits $\gamma_0,\gamma'_0$ such that
$\lambda(\gamma_0)\neq 0$ and $\lambda'(\gamma'_0)\neq 0$. For any periodic orbit $\gamma\subset X$,
one applies \Cref{lm:1} 
and build a symbolic system $h\colon \Sigma_r\to X$ such that
the periodic points $h(\overline i,0)$ lift points in each orbits $\gamma_0,\gamma'_0,\gamma$.
\Cref{p.hyperbolic-periodic} applied twice then ensures that there exist a periodic orbit
$\bar \gamma\subset h(\Sigma_r)$ whose Lyapunov exponents $\lambda(\bar\gamma), \lambda'(\bar \gamma)$
are non vanishing and close to $\lambda(\gamma),\lambda'(\gamma)$ as required.
\end{proof}

\Cref{t.main2} is now a consequence of the following.

\begin{Proposition}
If $\rho,\rho'$ are irreducible,
then for any periodic orbits $\gamma_1,\gamma_2,\gamma_3\subset X$
satisfying $\lambda(\gamma_i),\lambda'(\gamma_i)\neq 0$
we have
$$\bigg\{\sum a_i\cdot(\lambda(\gamma_i),\lambda'(\gamma_i)),\; a_1+a_2+a_3=1,\; a_i\geq 0\bigg\}\subset \Lambda.$$
\end{Proposition}

From \Cref{lm:1} and \Cref{l.extract-hyperbolic},
it reduces to the next statement.

\begin{Lemma}\label{p.symbolic}
Let $\Sigma_r$ be a suspension of the shift $\Sigma=\{1,2,3\}^\ZZ$
and let $(\mathcal{E},\mathcal{A}),(\mathcal{E}',\mathcal{A}')$
be cocycles associated to uniformly hyperbolic one-step locally constant cocycles $M,M'$.
Let $\lambda_i,\lambda'_i$ be the Lyapunov exponents
for $\mathcal{A},\mathcal{A}'$ of the periodic orbits
associated to the fixed points $\overline i\in \Sigma$.

Then, for any $a_1,a_2,a_3\in [0,1]$ with $a_1+a_2+a_3=1$,
there is an ergodic invariant probability measure $\mu$ on $\Sigma_r$
whose Lyapunov exponents for $\mathcal{A},\mathcal{A}'$ are
$\lambda=\sum a_i\lambda_i$ and $\lambda'=\sum a_i\lambda'_i$.
\end{Lemma}
\begin{proof}
The unstable bundles $E^u,{E^u}'$ for $M$ and $M'$ are continuous.
One can thus define the continuous functions
$$\psi(\underline \varepsilon)=\log \|M_{\varepsilon_0,\varepsilon_1}|_{E^u(\underline \varepsilon)}\|-\lambda {r(\underline \varepsilon)},\quad
\psi'(\underline \varepsilon)=\log \|M'_{\varepsilon_0,\varepsilon_1}|_{{E^u}'(\underline \varepsilon)}\|-\lambda' {r(\underline \varepsilon)}.$$
On fixed points, they take the values $\psi(\overline i)=(\lambda_i-\lambda ){r(\overline i)}$ and
$\psi'(\overline i)=(\lambda'_i-\lambda' ){r(\overline i)}.$ Let us set
$$\alpha_i=\tfrac{a_i}{r(\overline i)}\cdot \left[\tfrac{a(\overline 1)}{r(\overline 1)}+\tfrac{a(\overline 2)}{r(\overline 2)}+\tfrac{a(\overline 3)}{r(\overline 3)}\right]^{-1},$$
so that $\sum \alpha_i=1$, $\sum \alpha_i \psi(\overline i)=0$
and $\sum \alpha_i \psi'(\overline i)=0$.
The set of ergodic measures satisfies a weak convexity property, see \Cref{t.convexity}
in the appendix: there exists an ergodic measure $m$ on $\Sigma$ such that
$\int \psi dm=0$ and $\int \psi' dm=0$. In particular by definition of $\psi$, Birkhoff ergodic theorem gives for $m$-almost every $\underline \varepsilon\in \Sigma$:
\begin{equation*}\label{e.lambda}
\lim_n\frac 1 n {\sum_{k=0}^{n-1} \log\|M^1(\sigma^k(\underline \varepsilon))|_{E^u}\|}=\lambda \lim_n\frac 1 n {\sum_{k=0}^{n-1} r(\sigma^k(\underline \varepsilon))}.
\end{equation*}
The measure $m\times \RR$ defines an invariant finite measure on $\Sigma_r$. After normalization,
we denote it by $\mu$.
For $\mu$-almost every point $x$,
the quantity $\tfrac 1 t \log \| \mathcal{A}^t(x)|_{E^u}\|$ converges as $t\to +\infty$ to the Lyapunov exponent $\lambda(\mu)$ of $\mu$ for $\mathcal{A}$.
For $m$-almost every $\underline \varepsilon$ we set $x=(\underline \varepsilon,0)$ and get
$$\lambda=\lim_n\frac{\sum_{k=0}^{n-1} \log\|M^1(\sigma^k(\underline \varepsilon))|_{F}\|}{\sum_{k=0}^{n-1} r(\underline \varepsilon)}=\lim_t  \tfrac 1 t  \log \| \mathcal{A}^t(\underline \varepsilon, 0)|_{F}\|=\lambda(\mu).$$
The same argument shows that the Lyapunov exponent $\lambda'(\mu)$ of the measure $\mu$ for the cocycle $\mathcal{A}'$
coincides with $\lambda'$ as required.
\end{proof}

\section{Translation surfaces}
\label{sec:background}

In this section we review useful results concerning translation surfaces, their moduli spaces, affine invariant submanifolds and the Kontsevich--Zorich cocycle.
For a general introduction to translation surfaces and their moduli spaces, we refer the reader to the surveys~\cite{Forni:Matheus,Zorich:survey}.
\subsection{Moduli spaces of translation surfaces}

A structure of \emph{translation surface} on a compact oriented topological surface $(M,\Sigma)$ of genus $g$ is the data $(X,\omega)$ where $X$ is a genus $g$ Riemann surface together with a non identically zero holomorphic $1$-form $\omega$ where $\Sigma$ is the zero set of $\omega$. For any integer partition $\kappa$ of $2g-2$, let $\H(\kappa)$ denote the moduli space of nonzero holomorphic $1$-forms $(X,\omega)$ having prescribed zeroes of multiplicities $\kappa$. We also define the Teichm\"uller space $\mathcal{T}(\kappa)$ as the quotient of the set of structure of translation surfaces by the natural action of $\textrm{Diff}_0(M,\Sigma)$. Then, $\H(\kappa)$ is the quotient of $\mathcal{T}(\kappa)$ by the mapping-class group $\mathrm{MCG}(M,\Sigma)\coloneqq \textrm{Diff}^{+}(M,\Sigma) / \textrm{Diff}_0(M,\Sigma)$.

We use the upper script notation $\mathcal{T}^{(1)}(\kappa)$ and $\H^{(1)}(\kappa)$
for the space of unit area translation surfaces (\ie, $i/2\int \omega \wedge \overline{\omega} =1$).

The group $\mathrm{GL}^{+}_2(\RR)$ naturally acts on the set of structures of translation surface by post-composition with the translation charts. This action commutes with the action of
$\textrm{Diff}(M,\Sigma)$. Therefore it induces a well defined action on
both $\mathcal{T}(\kappa)$ and $\H(\kappa)$. We will refer to the action of the $1$-parameter diagonal subgroup
$g_t\coloneqq\diag(e^t,e^{-t})$ as the \emph{Teichm\"uller geodesic flow}.

\subsection{Affine invariant submanifolds}

Let $\gamma_1, \ldots, \gamma_n$ be any basis of the relative homology group $H_1(X, \Sigma, \ZZ)$,
where $n=2g+s-1$. The period coordinate maps defined by
$$
(X,\omega)\mapsto\left( \int_{\gamma_1} \omega, \ldots, \int_{\gamma_n}\omega \right)
$$
provide $\H(\kappa)$ with an atlas of charts to $\mathbb C^n$ with transition functions in $GL(n, \ZZ)$.
Over each stratum $\H(\kappa)$ there is a flat bundle $H^1$ whose fibers over $(X,\omega)$ are $H^1(X, \Sigma, \CC)$.
An \emph{affine invariant submanifold} of $\H(\kappa)$ is a properly immersed manifold $\mathcal N \hookrightarrow \H(\kappa)$ such that each point of $\mathcal N$ has a neighborhood whose image is given by the set of surfaces in an open set satisfying a set of real linear equations in period coordinates.
It follows directly from the definition that $\mathcal N$ is $\mathrm{GL}_2(\RR)$-invariant.

\subsection{The Kontsevich--Zorich cocycle}

On the trivial bundle $\mathcal{T}^{(1)}(\kappa) \times H^1(X,\mathbb R)$ we define
a trivial cocycle over $(g_t)$ as usual by
$$
\widehat{G_t}^{KZ}((X,\omega),c)=(g_t(X,\omega),c).
$$
We will often use the notation $H^1_\RR:=H^1(X,\mathbb R)$ for short. The mapping-class group $\mathrm{MCG}(M,\Sigma)$ acts in a non trivial canonical way on both factors of the
product $\mathcal{T}^{(1)}(\kappa) \times H^1(X,\mathbb R)$.
The quotient cocycle $G_t^{KZ}$ is the (reduced) Kontsevich--Zorich cocycle (or KZ cocycle for short). It acts on the real Hodge bundle $H^1_\RR$. 
We will denote by $G_t^{KZ}(X,\omega)$ the linear map from $H^1(X,\mathbb R)$ to $H^1(g_tX,\mathbb R)$.

\subsection{Lyapunov spectrum of the KZ cocycle}
Let $\Maff$ be a closed $\SL(2,\RR)$-invariant affine manifold.
Given any ergodic $(g_t)$-invariant probability measure $\mu$ with support contained in $\Maff^{(1)}$,
and any choice of norm $\|\cdot\|$ (\eg the Hodge norm), it follows from the work of~\cite{Forni}
that the KZ cocycle is integrable for $\mu$, \ie, $\int \log\|G_{\pm t}^{KZ}\|_{\mathrm{op}} d\mu<\infty$ for all $t\geq 0$.
The multiplicative ergodic theorem of Oseledets guarantees the existence of Lyapunov exponents
$\lambda_1^{\mu}>\dots>\lambda_k^{\mu}$ and a $G_t^{KZ}$-equivariant measurable decomposition $H^1(X,\RR) = E_1(\omega)\oplus\dots\oplus E_k(\omega)$ at $\mu$-almost every $(X,\omega)\in\Maff^{(1)}$ such that
$$
\lim\limits_{t\to\pm\infty} \frac{1}{t}\log\|G^{KZ}_t(X,\omega)v\| = \lambda_i^{\mu} \quad \forall \, \, v\in E_i(\omega)\setminus\{0\}.
$$
In general, we will write the Lyapunov exponent $\lambda_i^{\mu}$ with multiplicity $\dim E_i(\omega)$ in order to
obtain a list of $2g=\dim H^1(X,\RR)$ Lyapunov exponents.
Since $G_t^{KZ}$ is a symplectic cocycle (the action of $\mathrm{MCG}(M,\Sigma)$ on $H^1(X,\RR)$ preserves the natural symplectic intersection form $\{c,c'\}\coloneqq\int_X c\wedge c'$), the Lyapunov exponents are symmetric with respect to $0$.

By definition, $G_t^{KZ}$ acts on the tautological plane $H_{st}^1(X,\omega)\coloneqq\RR.\Re(\omega) \oplus \RR.\Im(\omega)\subset H^1(X,\RR)$ by the matrix $g_t=\diag(e^t, e^{-t})$.
This implies that $\pm 1$ are extremal Lyapunov exponents. Hence the Lyapunov spectrum is
$$
1=\lambda_1^{\mu}\geq\lambda_2^{\mu}\geq\dots\geq\lambda_g^{\mu}\geq-\lambda_g^{\mu}\geq\dots\geq-\lambda_2^{\mu}\geq-\lambda_1^{\mu}=-1.
$$

\subsection{Computing the KZ cocycle}
\label{sub:sec:computing}

In general, it is difficult to exhibit matrices of the KZ cocycle. However there is a
particular setting which allows to have a nice description: this is the case when the orbit closure of $(X,\omega)$
covers a \emph{Teichm\"uller curve}. One can then describe the KZ cocycle by the cohomological action of
\emph{affine homeomorphisms} of a Veech surface.

More concretely, the automorphism group ${\rm Aut} (X,\omega)$ (resp. the
affine group $\Aff (X,\omega)$) of $(X,\omega)$ is the group of
orientation-preserving homeomorphisms of $X$ which preserve $\Sigma$ and whose restrictions to
$X-\Sigma$ read as translations (resp. affine maps) in the translation charts. The
embedding of ${\rm Aut} (X, \omega)$ into $\Aff (X, \omega)$ is completed into an
exact sequence
$$
1 \longrightarrow {\rm Aut} (X,\omega) \longrightarrow \Aff (X,\omega)
\longrightarrow \SL(X,\omega) \longrightarrow 1 ,
$$
where $\Gamma:=\SL(X,\omega) \subset \SL(2,\RR)$ is the {\it Veech group} of $(X,\omega)$: it is the stabilizer
of $(X,\omega)$ under $\SL(2,\RR)$. The map from $\Aff (X,\omega)$ onto $\Gamma$ is defined by associating to each $\phi\in \Aff (X,\omega)$ its derivative (linear part) $D\phi\in\Gamma$ in translation charts.

When $(X,\omega)$ is a \emph{Veech surface}, that is, its Veech group is a lattice in $\SL(2,\RR)$,
its $\SL(2,\RR)$-orbit $\Maff$ is closed and is naturally isomorphic to $\SL(2,\RR)/\Gamma$.

The Teichm\"uller geodesic flow on $\Maff$ coincides with the geodesic flow on the finite-area hyperbolic surface
$\mathbb H/\Gamma$, that is with the (left) action of the diagonal group of matrices
$\left\{\left(\begin{smallmatrix} e^{t/2} & 0\\ 0 & e^{-t/2} \end{smallmatrix}\right),\; t \in \RR\right\}$
on $\PSL(2,\RR)/\Gamma$.
Moreover the affine group $\Aff(X,\omega)$ embeds naturally in the mapping-class group $\mathrm{MCG}(M,\Sigma)$
and its image is the stabilizer of $\Maff$ in $\mathcal{T}^{(1)}(\kappa)$.
Therefore, the restriction of the KZ cocycle to $\Maff$ is simply the quotient of the trivial cocycle
$$
G^{KZ}_t \colon \PSL(2,\RR)\times \RR^{2g} \to \mathcal \PSL(2,\RR)\times \RR^{2g}
$$
by the affine group, where we have identified $\RR^{2g}$ with $H^1(X,\RR)$.

\subsection{Lyapunov exponents and pseudo-Anosov maps}
\label{sub:pA}

An affine map $\phi \in \mathrm{Aff}(X,\omega)$ is a pseudo-Anosov map if $D\phi$
is a hyperbolic matrix. In this case, up to conjugacy and up to take the inverse, the derivative map has the form $D\phi = \left(\begin{smallmatrix} \theta_1 & 0\\ 0 & \theta_1^{-1} \end{smallmatrix}\right)$ where $|\theta_1|>1$. This determines a periodic orbit
(not necessarily simple) on $\H^{(1)}(\kappa)$ for $g_t$ of length $T=\log(|\theta_1|)$. By
definition of the KZ cocycle, one has $G^{KZ}_T(X,\omega)=\phi^{\ast}$ where $\phi^\ast$ is acting on $H^1(X,\RR)$.
Thus the eigenvalues $\theta^{\pm 1}_i$, $i=1,\dots,g$, of $\phi^{\ast}$ (where $|\theta_1|>|\theta_2|\geq \dots \geq |\theta_g|$)
are related to the Lyapunov exponents associated to ergodic $(g_t)$-invariant probability measure $\mu$ supported on the periodic orbit by
\begin{equation}
\label{eq:pA}
\lambda^{\mu}_i = \frac{\log(|\theta_i|)}{\log(|\theta_1|)}.
\end{equation}

\section{Diffusion rates and Lyapunov exponents}
\label{sec:le}

The main goal of this section is to relate the diffusion rate of the wind-tree model
to the Lyapunov exponents of the KZ cocycle. See~\cite{Delecroix:Hubert:Lelievre,Delecroix:Zorich} for more details.

\subsection{From wind-tree model to translation surfaces}
\label{sec:dim}
The following is a summary of the content of \cite[Section~3]{Delecroix:Hubert:Lelievre} needed for our purposes.

Let $a,b\in (0,1)$. The wind-tree model corresponds to a billiard in the plane
endowed with $\ZZ^2$-periodic, horizontally and vertically symmetries.
By the classical unfolding procedure, we can glue a translation surface $(X_\infty,\omega_\infty)$
out of four copies of the original billiard table and unwind billiard trajectories to flat geodesics on $X_\infty$.
The resulting surface $X_\infty$ is $\ZZ^2$-periodic with respect to translations by vectors of the original lattice.
Passing to the $\ZZ^2$-quotient we get a compact flat surface $(X_{a,b},\omega_{a,b}) \in \mathcal H(2,2,2,2)$ of genus~$5$, as in \Cref{f:compact-surface}.
For simplicity, we denote it by $(X,\omega)$.
The $\ZZ^2$ covering $X_\infty$ over $X$ is defined by the Poincaré dual $f \in H^1(X,\ZZ^2)$ of the cycle
\begin{equation}
\label{eq:cocycle}
\begin{pmatrix} v \\ h \end{pmatrix} = \begin{pmatrix} v_{00} + v_{01} - v_{10} - v_{11} \\ h_{00} - h_{01} + h_{10} - h_{11} \end{pmatrix} \in H_1(X,\ZZ^2) \cong H_1(X,\ZZ)^2,
\end{equation}
where the cycles $h_{i,j}$ and $v_{i,j}$, $i,j \in \{0,1\}$, are as in \Cref{f:compact-surface}.

\begin{figure}[hbt]
\includegraphics[scale=.5]{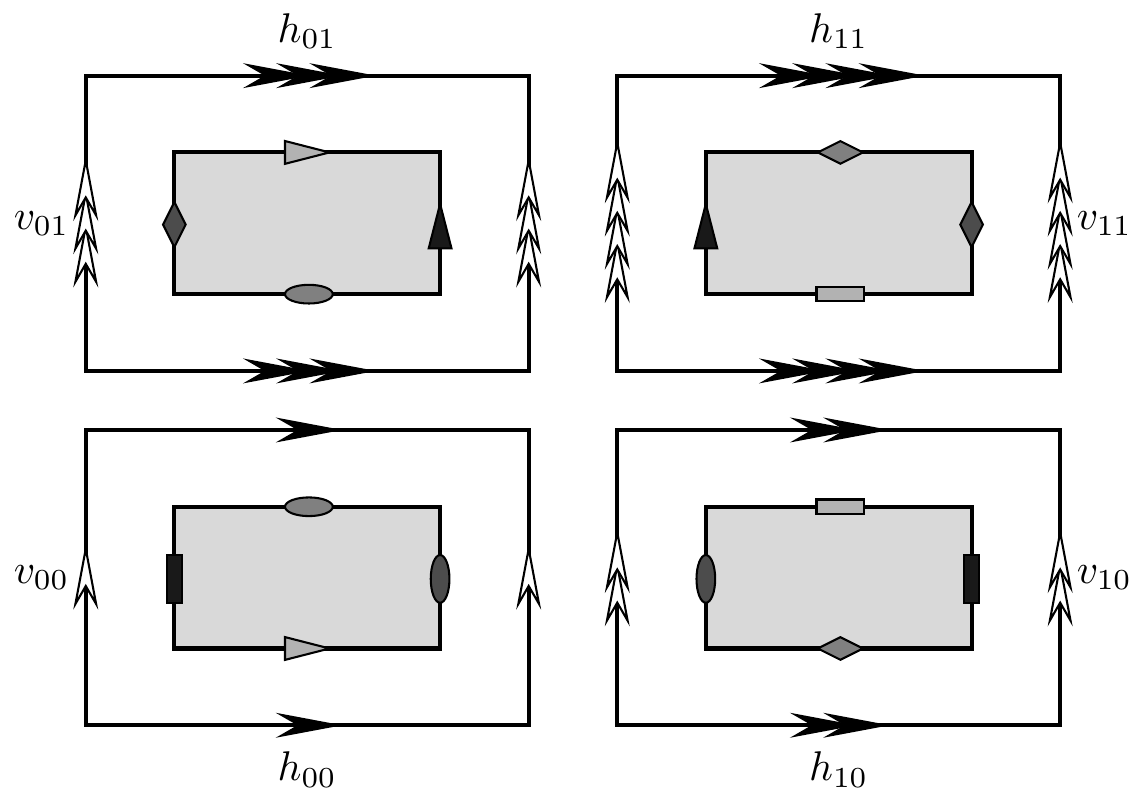}
\caption{The compact translation surface $(X, \omega)$ obtained as quotient over $\ZZ^2$ of an unfolded wind-tree billiard table (\cf \cite[Figure 5]{Delecroix:Zorich}).}
\label{f:compact-surface}
\end{figure}

Observe that $(\ZZ_2)^2$ is a subgroup of $\Aut(X,\omega)$: it is generated by two isometries $\tau_h,\tau_v$. More precisely (see \Cref{f:compact-surface}), $\tau_h$ interchanges the pairs of flat tori with holes in the same rows by parallel translations while the isometry $\tau_v$, the flat tori with holes in the same columns.

This produces a flat symplectic decomposition of the bundle
\begin{equation}
\label{eq:splitting}
H^1(X,\RR)=E^{++}\oplus E^{+-}\oplus E^{-+}\oplus E^{--},
\end{equation}
over $\Maff$, the $\GL(2,\RR)$ orbit closure of $(X,\omega)$ in $\H(2,2,2,2)$,
where $E^{++}$ is the vector space invariant by $\tau_h$ and $\tau_v$, $E^{+-}$ the vector space invariant by $\tau_h$ and
anti-invariant by $\tau_v$, etc. In the sequel we will simply denote $E^{+-}$ by $E^{+}$ and $E^{-+}$ by $E^{-}$.
It follows that we have $E^{++} \simeq H^1(X/\langle\tau_h,\tau_v\rangle,\RR)$ and thus, $H_{st}^1(X,\omega)\subset E^{++}$. Moreover $E^{++}$ is real $4$-dimensional since $X/\langle\tau_h,\tau_v\rangle \in \H(2)$ has genus $2$.
We can also check that the other three subbundles are real $2$-dimensional.
Note that the cocycle $f={(h,v)}^* \in H_1(X,\ZZ^2)$ defining the $\ZZ^2$-cover satisfy $h^* \in E^+$ and $v^* \in E^-$, where ${(\cdot)}^*$ denotes Poincaré duality.

The splitting~\eqref{eq:splitting} extends as a constant splitting of the trivial bundle $\mathrm{PSL}(2,\RR) \times H^1(X,\RR)$ over the $\SL(2,\RR)$-orbit of $(X,\omega)\in \mathcal{T}^{(1)}(\kappa)$. The quotient by the action of the affine group $\Aff (X,\omega) \subset \mathrm{MCG}(M,\Sigma)$
is a well-defined splitting of the real Hodge bundle $H^1(X,\RR)$ over $\Maff^{(1)}$, which is invariant under the KZ cocycle.\footnote{To be very precise, one should consider orbifold vector bundles instead of vector bundles since the spitting is not invariant by the full group $\Aff (X,\omega)$ but only by a subgroup of finite index. Since it does not affect the rest of the paper, we will not enter this technicality. }

\begin{Remark}
\label{r:DZ-decomposition}
In the case of the Delecroix--Zorich variant, the corresponding subspaces to consider are $E^{+} \simeq H^1(X/\langle\tau_h,\iota\circ\tau_v\rangle,\RR)$ and $E^{-} \simeq H^1(X/\langle\iota\circ\tau_h,\tau_v\rangle,\RR)$, where $\iota$ is the involution that rotates in $\pi$ each one of the four copies of the fundamental domain in $X$ (flat tori with holes).
In this case, the quotient surfaces are half-translation surfaces of genus one and therefore $E^+$ and $E^-$ are still (virtually) $\Aff(X,\omega)$-invariant real $2$-dimensional subspaces of $H^1(X,\RR)$ and symplectic-orthogonal to $H_{st}^1(X,\omega)$.
Together with the fact that the underlying surface $X$ is a Veech surface (since the side-lengths are rational), these are the only conditions actually needed to our result to hold.
See \cite[Section~3.2]{Delecroix:Zorich} and \cite[Section~2.4]{Pardo:counting} for further details.
\end{Remark}

\subsection{The Kontsevich--Zorich cocycle as a suspension of a representation}
\label{sec:interval}

We now assume $(a,b)\in \mathcal E$. By~\cite{McMullen}, the surface $(X, \omega) = (X_{a,b},\omega_{a,b})$ is a Veech surface. We denote by $\Gamma$ its Veech group. This group does not necessarily preserve the bundles $E^{\pm}$, but there is a finite index subgroup $\Gamma_0$ fixing globally each sub-bundle. The Teichm\"uller geodesic flow on
$\Maff^{(1)}$ coincide with the geodesic flow on the unit bundle of the finite-area hyperbolic surface $\mathbb H/\Gamma$. There are two representations $\rho^{-}$ (resp. $\rho^{+}$) of $\Gamma_0$ into $\Aut(E^{-})$ (resp. $\Aut(E^{+})$) given by the restriction of the associated affine map on the subbundle $E^{+}$ (resp. $E^{-}$). We denote by $A^t$ the restriction of the KZ cocycle to the subbundle $E^{+}\oplus E^{-}$: it is obtained as the product of the suspensions of the representations $\rho^{+},\rho^{-}$. In the sequel we will denote $\rho=(\rho^{+},\rho^{-})$.

Let $\mathcal M$ be the set of ergodic $(g_t)$-invariant probability measure $\mu$ whose support is contained in $\Maff^{(1)}$.
Given any ergodic $\mu\in \mathcal{M}$ whose support is contained in $\Maff^{(1)}$, we let $\lambda^{+}(\mu)$ and $\lambda^{-}(\mu)$ be the two non negative Lyapunov exponents of $A^t$. We set $\lambda(\mu) =\max\{\lambda^{+}(\mu),\lambda^{-}(\mu)\}$
and $\Lambda=\{(\lambda^+(\mu),\lambda^-(\mu)),\mu \in \mathcal M\}$.

\subsection{Lyapunov exponents and diffusion rate}

The next result is proved in~\cite{Delecroix:Hubert:Lelievre} only for $\SL(2,\RR)$-invariant measure.
However, it is possible to strengthen the conclusion to any $(g_t)$-invariant measure.

\begin{Theorem}
\label{thm:equality:exp:diff}
For any $(a,b)\in \mathcal E$, for any ergodic $(g_t)$-invariant probability measure $\mu$ on $\mathcal{N}^{(1)}$, there
exists $\theta\in\mathcal{D}$ such that $\delta_\theta = \lambda(\mu)$.
\end{Theorem}
\begin{proof}
For any translation surface $(Y,\eta)\in \mathcal H(\kappa)$, and
any $t \in \RR$, $y\in Y$ we define $\gamma_t(Y,y)\in H_1(Y,\ZZ)$
a vertical geodesic segment of length $t$ from $y$ that is closed by a
uniformly bounded curve. Let $f\in H^1(X,\ZZ^2)$ be the cocycle in \eqref{eq:cocycle} that defines the wind-tree model. Then~\cite[Proposition~1]{Delecroix:Hubert:Lelievre} establishes that for a given $\theta$ and $x$
$$
\delta_\theta(x) = \limsup_{T \to +\infty} \frac{\log \|\langle f,\gamma_T(r_\theta X,x)\rangle\|}{\log T},
$$
where $\gamma_T(r_\theta X,x)\in H_1(X,\ZZ)$.
Given a small neighbourhood $U$ of $(X,\omega)$ in $\mathcal{N}^{(1)}$, for which there is a trivialization of the Hodge bundle, for any $(Y,\eta)\in U$ one naturally defines, for $y\in Y$
$$
F(Y,y)\coloneqq\limsup_{T \to +\infty} \frac{\log \|\langle f,\gamma_T(Y,y)\rangle\|}{\log T}.
$$
Observe that $F(Y,y)$ is constant (equal to $F(Y)$) for Lebesgue almost every $y\in Y$. A quick inspection of the proof of~\cite[Lemma~11]{Delecroix:Hubert:Lelievre}
(related to Oseledets theorem) reveals that in fact its conclusion holds for any
$(g_t)$-invariant ergodic measure $\mu$. Namely, for $\mu$ almost every $(Y,\eta)\in U$
\begin{equation}
\label{eq:la2}
F(Y)= \lambda(\mu).
\end{equation}
(The proof is identically to the one of Delecroix--Hubert--Leli\`evre. The additional step to check is
that the restriction of $f$ to $E^{+}\oplus E^{-} \subset H^1(Y,\mathbb R)$ does not belong to the stable subspace.
But this restriction is an integer covector and the KZ cocycle takes values in the set of integer matrices
of determinant $1$ so $\|G^{KZ}_t(Y)f\|$ cannot go to zero as $t$ goes to infinity as it is
bounded from below by a positive constant). Hence~\eqref{eq:la2} holds. \medskip

Next we want to rely~\eqref{eq:la2} to the diffusion rate $\delta_\theta(x)$.
We follow ideas in a previous version of~\cite{Delecroix:Hubert:Lelievre} (precisely, \cite[Section~6]{Delecroix:Hubert:Lelievre:arXiv}, in the third arXiv version of \cite{Delecroix:Hubert:Lelievre}), prior to~\cite{Chaika:Eskin} .
Since the support of $\mu$ is contained in $\Maff^{(1)}$, $U$ can be locally identified with a neighbourhood of
$\SL(2,\RR)$ (recall that $(a,b)\in \mathcal E$, so $X_{a,b}$ is a Veech surface). The Iwasawa decomposition implies that the set
$$
\{h_sg_tr_\theta X; s,t\in \RR,\ \theta\in \mathbb S^1 \}
$$
contains a generic point $(s,t,\theta)$ of the measure $\mu$. Hence $\lambda(\mu)=F(Y)=F(h_sg_tr_\theta X)$.
But the function $F(\cdot)$ is invariant under the geodesic flow and the horocyclic flow.
It follows that $\lambda(\mu)=F(h_sg_tr_\theta X,y)=F(r_\theta X,h_sg_t(y))=\delta_\theta(x)$, where $x=h_sg_t(y)$ is the affine action of $h_sg_t$ on $Y$. This holds for Lebesgue almost every $x\in X$ thus $\theta\in \mathcal{D}$ and $\delta_\theta(x) = \delta_\theta=\lambda(\mu)$ as required.
\end{proof}

We end this section with two remarks.

\begin{Remark}
\label{rk:join:diffusion}
We can easily extend the previous theorem to join diffusions. Namely, by considering $d_{\mathrm{h}}$ and $d_{\mathrm{v}}$ instead of the distance $d$, we construct the function $F$ and show that $F(Y)$ is equal to $(\lambda(\mu),\lambda'(\mu))$ for $\mu$ almost every $(Y,\eta)$. We then use~\cite[Section~6]{Delecroix:Hubert:Lelievre:arXiv} to rely this quantity to the join diffusion $(\delta^{\mathrm{h}}_\theta, \delta^{\mathrm{v}}_\theta)$ for some $\theta$. More precisely, the following result holds: For any $(a,b)\in \mathcal E$ and any ergodic $(g_t)$-invariant probability measure $\mu$ on $\mathcal{N}^{(1)}$, there exists $\theta\in\mathcal{D}$ such that $(\delta^{\mathrm{h}}_\theta, \delta^{\mathrm{v}}_\theta) = (\lambda(\mu),\lambda'(\mu))$.
\end{Remark}

\begin{Remark}
\label{rk:top}
Forni~\cite{Forni} showed that the maximal Lyapunov exponent associated to any ergodic $(g_t)$-invariant
measure is always simple.
In our situation the KZ cocycle restricted to the subbundle whose fiber above $(X,\omega)$ is $E^{++}\oplus E^{+} \oplus E^{-}$ has non
negative spectrum
$$
1,\lambda^{++}(\mu), \lambda^{+}(\mu),\lambda^{-}(\mu).
$$
Since the maximal Lyapunov exponent is $1$ we have
$\lambda(\mu)=\max\{\lambda^{+}(\mu),\lambda^{-}(\mu)\}<1$.
\end{Remark}

\section{Parabolic elements}
\label{sec:parabolic}

In order to prove \Cref{thm:main} and \Cref{thm:square} we use the existence of some parabolic elements in the Veech group outside the kernel of the representations $\rho^{+}$ and $\rho^{-}$.
In this direction, in this section we give a geometric criterion in order to classify parabolic elements in the kernel of these representations.

\subsection{Parabolic elements and cylinder decompositions}
A \emph{cylinder} on a translation surface $(X,\omega)$ is a maximal open annulus filled by isotopic simple closed regular geodesics, isometric to $\RR/w\ZZ \times (0,h)$, for some $w,h > 0$.
Its \emph{modulus} is the ratio $m = w/h$ and its \emph{core curve} is the simple closed geodesic identified with $\RR/w\ZZ \times \{h/2\}$.
A \emph{cylinder decomposition} of $X$ is a collection of parallel cylinders with disjoint interiors and whose closures cover $X$.
The direction of a cylinder or a cylinder decomposition corresponds to the angle between the horizontal direction and the core curves.

The following result is needed for what follows.
\begin{Theorem}\label{thm:cylWei}
Let $(a,b) \in \mathcal{E}$. Then, for any pair of regular Weierstrass points in $L_{a,b}$ there is a cylinder in $L_{a,b}$ whose core curve passes through these two points.
\end{Theorem}

\begin{proof}
This is a direct corollary of \cite[Theorem~1]{Pardo:non-varying}, which states that the number of cylinders whose core curve passes through any pair of marked regular Weierstrass points has quadratic growth rate (positive \emph{area Siegel--Veech constant}).
In particular, there are infinitely many such cylinders for any pair of regular Weierstrass points in $L_{a,b}$.
\end{proof}

By the work of Veech~\cite{Veech}, cylinder decompositions are related to parabolic elements in the following way.
Let $R_\theta$ be the matrix rotating the plane (counterclockwise) by $\theta \in [0,2\pi)$ and let $T^t = \left(\begin{smallmatrix} 1 & t \\ 0 & 1 \end{smallmatrix}\right)$, $t \in \RR$.
Then, a matrix of the form $P_\theta^t = R_\theta T^t R_\theta^{-1}$ belongs to $\SL(X,\omega)$ if and only if there exists a cylinder decomposition $\{C_i\}_{i \in I}$ of $X$ in direction $\theta$ such that the moduli $m_i$ of the cylinders $C_i$ are pairwise rationally related, that is, if $r_{ij} = \frac{m_i}{m_j} \in \QQ$ for every $i,j \in I$.
Furthermore, if $t \in \RR$ is an integer multiple of the modulus of each cylinder, that is, for each $i \in I$ there exist $k_i \in \ZZ$ such that $t = k_im_i$, then $P_\theta^t \in \SL(X,\omega)$. The corresponding affine transformation preserves the cylinder decomposition, as it twists each cylinder $C_i$ exactly $k_i$ times along itself.

\subsection{Geometric criterion}
\label{para-geom}
The aim of this section is to prove a geometric criterion for a parabolic element in $\Gamma$ to be in $K^{\pm} \coloneqq \ker(\rho^{\pm})$ (\Cref{thm:geom} below).
But we first give the actual geometric intuition and motivation 
in terms of the symmetries of the wind-tree model.
These ideas date back to \cite{Hubert:Lelievre:Troubetzkoy} and have been used in several other works on the wind-tree model, in particular, in \cite{Avila:Hubert} to provide a geometric criterion for the recurrence of the model or further, in \cite{Pardo:counting}, to prove asymptotic formulas for the number of periodic trajectories.
Even if we do not use explicitly this approach in the proofs, this is very much the actual ideas behind them.
We supply them for the reader's convenience.

Recall that a wind-tree billiard covers an L-shaped translation surface $L_{a,b} = X_{a,b}/\langle\tau_h,\tau_v\rangle \in \H(2)$, for some $a,b \in (0,1)$, as in \Cref{fig:Lab} (left).
The surface $L = L_{a,b}$ is thus hyperelliptic and its six Weierstrass points are labeled $A,B,C,D,E,F$ as depicted in \Cref{fig:Lab} (right).
When the parameters $a, b$ are rational, the surface $L$ is 
a square-tiled surface.

\begin{figure}[ht]
\includegraphics{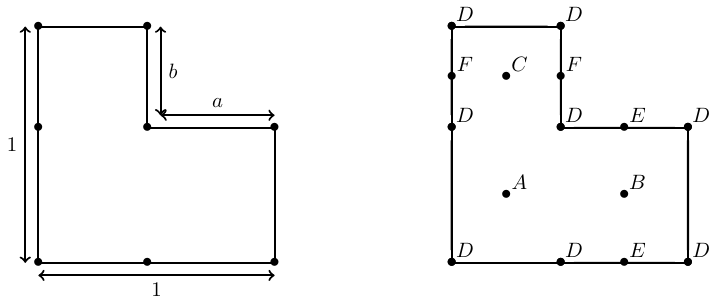}
\caption{The surface $L_{a,b}$ (left) and its six Weierstrass points (right).}
\label{fig:Lab}
\end{figure}

In $\H(2)$, as in any hyperelliptic component, every cylinder is invariant under the hyperelliptic involution and, in particular, the core curve of any cylinder contains exactly two regular Weierstrass points (see, \eg, \cite[Remark~5]{Pardo:non-varying}).
At the level of the wind-tree model, points in the fibers over the five regular Weierstrass points $A,B,C,E,F \in L$ give raise to symmetries of the infinite billiard table.
Thus, as can be seen in \Cref{fig:wtm-symmetries}, the points over $A,B,C$ give central symmetries, while the points over $E$ (resp. $F$) give axial symmetries through the corresponding vertical (resp. horizontal) lines passing through.

\begin{figure}[ht]
\includegraphics{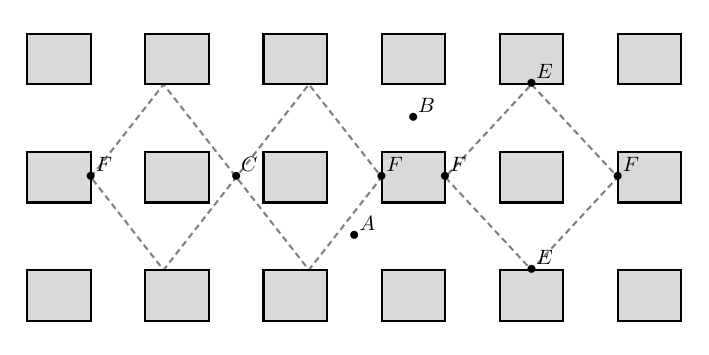}
\caption{The wind-tree model and some points over the regular Weierstrass points of $L \in \H(2)$ labeled as their image in $L$.
The points over $A,B,C$ give central symmetries.
The point over $E$ (resp. $F$) gives an axial symmetry through the corresponding vertical (resp. horizontal) line passing through.
}
\label{fig:wtm-symmetries}
\end{figure}

Because of these symmetries, 
it is easy to see, for example, that a closed geodesic passing though the point $A$ is necessarily unbounded in the corresponding wind-tree billiard. Moreover, by the same reasons, the only possibilities for a closed geodesic in $L$ passing through two regular Weierstrass points to lift into a periodic trajectory in the wind-tree model is that either
\begin{itemize}
\item it passes through $E$ and $F$; or
\item it passes through $E$ and $B$ (resp. $F$ and $C$) in such a way that the corresponding points in the fiber in the wind-tree model are in the same vertical (resp. horizontal) line;
\end{itemize}
and any such trajectory is indeed periodic.
See \Cref{fig:wtm-symmetries}.
However, the last case can be easily broken through the action of the affine group.

On the other hand, parabolic elements in the Veech group of $L$ acts on the homology of the wind-tree model as an affine multi-twists through the core curves of the cylinders in the corresponding direction.
Thus, in order to act trivially in $E^{+}$ and $E^{-}$, the corresponding affine multi-twists have to, in particular, fix the cocycles $h\in E^{+}$ and $v\in E^{-}$ defining the $\ZZ^2$-covering. In other words, the core curves have to have trivial monodromy or, what is the same, give raise to periodic trajectories in the wind-tree model (closed curves in the $\ZZ^2$-cover).
Moreover, the spaces $E^{+}$ and $E^{-}$ are virtually invariant under the action of the affine group and this action is irreducible.
It follows from the previous discussion that the only possibility left for a parabolic element to be in the kernel of both representations is that the core curves of the cylinders in the corresponding direction passes through $E$ and $F$.

The previous discussion can be formalized (and refined) as the following geometric criterion.

\begin{Theorem}
\label{thm:geom}
Let $p \in \Gamma$ be a parabolic element. Then, $p \in K^{+}$ \emph{(resp. $K^{-}$)} if and only if the direction fixed by $p$ is a one-cylinder direction in $L$ and $E$ \emph{(resp. $F$)} is in the core curve of that cylinder.
\end{Theorem}

Before proving \Cref{thm:geom}, we give the following
two results
needed for the proof of \Cref{thm:main}.
The first, on the existence of parabolic elements outside of the kernel of both representations, allows to ensure that $0 \in \Lambda$.

\begin{Corollary}
\label{c:parabolic-in-complement}
Let $\Gamma_0$ be any finite index subgroup of $\Gamma$.
Then, there is a parabolic element $\gamma \in \Gamma_0$ not in either $K^{+} = \ker(\rho^{+})$ or $K^{-} = \ker(\rho^{-})$.
\end{Corollary}

\begin{proof}
Take any two-cylinder decomposition of $L$.
Let $p \in \Gamma$ be the parabolic element associated to the corresponding cylinder decomposition on $X$.
Since $\Gamma_0$ has finite index in $\Gamma$, some power of $p$ is in $\Gamma_0$.
But, by \Cref{thm:geom}, no power of $p$ can be in $K^{+}$ or $K^{-}$.
\end{proof}

\begin{Remark}
In the case of the Delecroix--Zorich variant, a geometric description of parabolic elements in the kernels of the representation is also possible but significantly more intricate (\cf \cite[Section~6.1]{Pardo:counting}).
Nonetheless, \Cref{c:parabolic-in-complement} is still valid for the Delecroix--Zorich variant.
In fact, it is enough to exhibit one such parabolic element and it is possible to verify that the horizontal and the vertical directions give parabolic elements that are in neither of the representations.
\end{Remark}

The following result shows the existence of \emph{strip decompositions}, ensuring that almost every trajectory in that direction escapes linearly to infinity, that is, $1 \in \{\delta_\theta;\, \theta \in \mathcal D\}$.
More precisely, a \emph{strip} on a (necessarily infinite) translation surface is an isometrically embedded product of an open interval and a straight line.
In \Cref{fig:strips} we depict several strips embedded in the wind-tree billiard.

In our context, strips arises as lifts of cylinders having non-trivial $\ZZ^2$-monodromy.
Thus, for example, the horizontal and vertical foliations on $X_\infty$ decomposes as a reunion of (the closure of) strips and cylinders.
A \emph{strip decomposition} of $X_\infty$ is then a collection of parallel strips whose closures cover $X_\infty$.

Now, if $X_\infty$ allows a strip decomposition in a direction $\theta \in \mathbb S^1$, it is clear that $\theta \in \mathcal D$ and $\delta_\theta = 1$ as any trajectory inside a strip escapes linearly.

\begin{figure}[ht]
\includegraphics{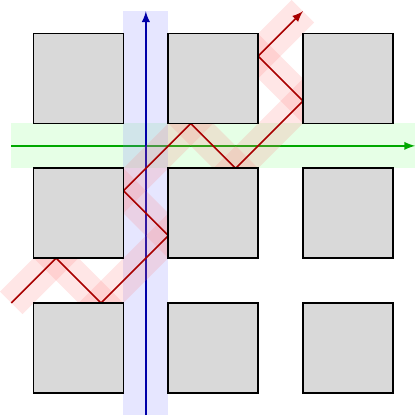}
\caption{Some cylinders in $X$ lift to strips in $X_\infty$. They correspond to unbounded \emph{drift-periodic} billiard trajectories in the wind-tree model.}
\label{fig:strips}
\end{figure}

\begin{Proposition}
\label{c:strip-decomposition}
For any $(a,b) \in \mathcal E$, there is a strip decomposition on $X_\infty = X_\infty(a,b)$.
\end{Proposition}

\begin{proof}
By \Cref{thm:cylWei}, there is a cylinder in $L = L_{a,b}$ whose core curve passes through $E$ and $C$.
By our discussion on the symmetries of the model, it is clear that such a cylinder lifts as a strip on $X_\infty$.
If this cylinder determines a one-cylinder decomposition of $L$, then we get a strip decomposition of $X_\infty$ and we are done.
Otherwise, it belongs to a two-cylinder decomposition.
But then, the core curve of the other cylinder passes either by $\{A,B\}$, $\{A,F\}$ or $\{B,F\}$.
Again, by the symmetries of the model, it lifts as a strip on $X_\infty$ in any case, and we  get a strip decomposition of $X_\infty$.
\end{proof}

\subsection{Proof of the geometric criterion \Cref{thm:geom}}
We need first the two following results.
\begin{Lemma} \label{lemm:ann}
The core curve $\gamma_C$ of a cylinder $C \subset X$ is in $\ann(E^{+})$ \emph{(resp. $\ann(E^{-})$)} if and only if it passes through some preimage of $E$ \emph{(resp. $F$)}.
\end{Lemma}
\begin{proof}
By \cite[Proposition~6.3]{Pardo:counting}, the core curve $\gamma_C$ of a cylinder $C \subset X$ is in $\ann(E^{+})$ if and only if its projection in $L$ passes through at most one of the points $A,B,C,F$. In particular, since the core curve of any cylinder in a surface in $\H(2)$ passes through exactly two regular Weierstrass points, this happens if and only if $\gamma_C$ passes through some preimage of $E \in L$.
Similarly, $\gamma_C \in \ann(E^{-})$ if and only if $\gamma_C$ passes through some preimage of $F \in L$.
\end{proof}

\begin{Lemma}\label{lemm:twists}
Let $\gamma$ be the core curve of a cylinder in $L$ and $(\hat \gamma_i)_i$ be the lifts of $\gamma$ to $X$.
Then, the multi-twist defined (in cohomology) by $(\hat \gamma_i)_i$ and $\alpha \neq 0$ as
\begin{align*}
T \colon H^1(X) & \to H^1(X) \\ 
 f & \mapsto f + \alpha \sum_i f([\hat \gamma_i]) [\hat \gamma_i]^*
\end{align*}
acts trivially in $E^{+}$ \emph{(resp. $E^{-}$)} if and only if $E \in \gamma$ \emph{(resp. $F \in \gamma$)}.
\end{Lemma}
\begin{proof}
If $E \in \gamma$, by \Cref{lemm:ann}, $\hat \gamma_i \in \ann(E^{+})$ and $T$ acts trivially in $E^{+}$.

For the converse, let $n$ be the number of lifts of $\gamma$ to $X$
and $\hat \gamma \subset X$ be one of those lifts.
By \Cref{lemm:ann} again, if $E \notin \gamma$, then there exists $f \in E^{+}$ such that $f([\hat \gamma]) \neq 0$.
Moreover, it is clear that $\{\hat \gamma_i\}_i = \langle \tau_h, \tau_v\rangle . \hat \gamma$ and $\# \stab_{\langle \tau_h, \tau_v\rangle}(\hat \gamma) = 4/n$.
Using this and the fact that $f \in E^{+}$, it follows that
\begin{align*}
T(f)
 & = f + \alpha \frac{n}{4}\big( f([\hat \gamma]) [\hat \gamma]^* + f([\tau_h \hat \gamma]) [\tau_h \hat \gamma]^* + f([\tau_v \hat \gamma]) [\tau_v \hat \gamma]^* + f([\tau_h \tau_v \hat \gamma]) [\tau_h \tau_v \hat \gamma]^*\big) \\
 & = f + \alpha nf([\hat \gamma]) \frac{1}{4}\big([\hat \gamma]^* + [\tau_h \hat \gamma]^* - [\tau_v \hat \gamma]^* - [\tau_h \tau_v \hat \gamma]^*\big) \\
 & = f + \alpha nf([\hat \gamma]) \pr_{E^{+}}([\hat \gamma]^*).
\end{align*}
But $\hat \gamma \notin \ann(E^{+})$, $\pr_{E^{+}}([\hat \gamma]^*) \neq 0$ and, by the choice of $f \in E^{+}$, $f([\hat \gamma]) \neq 0$.
It follows that $T(f) \neq f$ and, thus, $T$ does not act trivially in $E^{+}$.
\end{proof}

\begin{proof}[Proof of \Cref{thm:geom}]
Let $p \in \Gamma$ be a parabolic element.
Then $p$ acts as an affine multi-twist $T_p$ along a cylinder decomposition in $X$.
That cylinder decomposition descends to $L \in \H(2)$.

Suppose that we get a one-cylinder decomposition in $L$, say by $C$. Then, the multi-twist $T_p$ in $X$ acts in cohomolgy as the multi-twist defined in \Cref{lemm:twists}, with $\alpha = 1$ as every lift of the one-cylinder in $L$ has the same modulus.
By \Cref{lemm:twists}, this action is trivial if and only if $E \in \gamma$, the core curve of $C$.
That is to say, for one-cylinder decompositions in $L$, $p \in K^{+}$ if and only if $E$ is in the core curve of the one-cylinder.

Suppose now that we get a two-cylinder decomposition in $L$, say by $C_i$, $i=0,1$.
Then, (in cohomology) the multi-twist $T_p$ in $X$ is the product of two (commutative) multi-twist as in \Cref{lemm:twists}, say $T_i$, corresponding to a twist along the core curve of $C_i$, with appropriate $\alpha_i > 0$, for $i=0,1$. Thus $T_p = T_1 T_0 = T_0T_1$.
\begin{itemize}
\item Suppose $E$ is in the core curve of one of the two cylinders, say in $C_0$.
Then, by \Cref{lemm:twists}, $T_0$ acts trivially in $E^{+}$ and $T_p|_{E^{+}} = T_1|_{E^{+}}$ which is non-trivial since $E$ cannot be in the core curve of $C_1$ as well.
In particular, $p \notin K^{+}$.
\item Suppose $E$ is not in any of the core curves $\gamma^{(i)}$ of $C_i$, $i=0,1$.
It follows, by \Cref{lemm:ann}, that any lift $\hat \gamma^{(i)}$ of $\gamma^{(i)}$ to $X$ does not belong to $\ann(E^{+})$, that is, $f_i = \pr_{E^{+}}[\hat \gamma^{(i)}]^* \neq 0$, $f_i \in E^{+}$, $i=0,1$.
Note that for $f \in E^{+}$, we have $f([\hat\gamma^{(i)}]) = \langle f, f_i \rangle$. Then, analogously to the proof of \Cref{lemm:twists} we get that, for $f \in E^{+}$,
\[ T_p(f) = T_1 T_0(f) = f + \alpha_0 n_0 \langle f, f_0 \rangle f_0 + \alpha_1 n_1 \langle f, f_1 \rangle f_1.\]
\begin{itemize}
\item If $f_0$ and $f_1$ are not collinear, since $E^{+}$ is symplectic and $\dim E^{+} = 2$, we get that $\langle f_1, f_0 \rangle \neq 0$.
Then, $T_p(f_1) = f_1 + \alpha_1 n_1 \langle f_1, f_0 \rangle f_0 \neq f_1$.
\item If $f_0$ and $f_1$ are collinear, consider $0 \neq f = \lambda_0 f_0 = \lambda_1 f_1 \in E^{+}$. Since $E^{+}$ is symplectic, it follows that there exists $f' \in E^{+}$ with $\langle f, f' \rangle \neq 0$. Thus,
\[ T_p(f') 
 = f' + \left(\frac{\alpha_0 n_0}{\lambda_0^2} + \frac{\alpha_1 n_1}{\lambda_1^2} \right) \langle f, f' \rangle f.\]
But $\alpha_i > 0$ and $n_i \in \N$,
for $i=0,1$.
It follows that $T_p(f') \neq f'$.
\end{itemize}
Thus, in both cases, $T_p$ does not act trivially in $E^{+}$ and therefore, $p \notin K^{+}$.
\end{itemize}
It follows that for two-cylinder decompositions in $L$, $p$ is never in $K^{+}$.
\end{proof}


\section{Endpoints} 
\label{sec:endpoints}

In this section we conclude the proof of \Cref{thm:main} by showing that $(0,0) \in \Lambda$ and $(1,1) \in \overline\Lambda$.
We will first establish a result on the kernel of the representations $\rho^{+}$ and $\rho^{-}$.

\begin{Theorem}
\label{thm:kernel}
Let $H_1(X,\RR)=E\oplus W$ be a virtually $\mathrm{Aff}(X,\omega)$-invariant splitting
such that $H_{st}^1(X,\omega) \subset E$.
If $\mathrm{rank}(W)=2$,
then the representation $\rho \colon \Gamma\rightarrow \Aut(W)$
is not faithful.
Moreover, the limit set of the Fuchsian group $\ker(\rho)$ is the full circle at infinity.
\end{Theorem}
\begin{proof}
The proof follows the same line of ideas as~\cite[Theorem 5.5 and Theorem 5.6]{Hopper:Weiss}.
\end{proof}

Since the representation $\rho^{+}$ (resp. $\rho^{-}$) takes values in $\Aut(E^{+})=\SL(2,\ZZ)$ (resp. $\Aut(E^{-})=\SL(2,\ZZ)$) (see \Cref{sec:dim} and~\eqref{eq:splitting}), we have the following.
\begin{Corollary}
\label{cor:not:faithful}
The limit sets of $\ker(\rho^{+})\subset \Gamma$ and $\ker(\rho^{-})\subset \Gamma$ is the full circle at infinity.
\end{Corollary}
\begin{proof}[Proof of \Cref{thm:main}]
By \Cref{thm:equality:exp:diff} we have that the set
$$\Lambda^\max=\{\max\{\lambda^{+}(\mu),\lambda^{-}(\mu)\}, \; \mu \in \mathcal M\}$$
satisfies $\Lambda^\max\subset\left\{\delta_\theta; \ \theta\in \mathcal D \right\}$.
The representations $\rho^{+}$ and $\rho^{-}$ of the wind-tree model are irreducible (see, \eg, \cite[Section~4.2.1]{Pardo:non-varying}, where we show that their images are non-elementary and, in particular, strongly irreducible and Zariski dense). 
Thus, by \Cref{t.main2}, the set
$$\Lambda = \{(\lambda^{+}(\mu),\lambda^{-}(\mu)), \; \mu \in \mathcal M\}$$
has convex relative interior and therefore $\Lambda^\max$ is an interval.
\Cref{rk:top} shows that $1\not \in \Lambda^\max$.
Let us prove that $\Lambda^\max=[0,1)$.

\begin{description}
\item[\hypertarget{left-end}{Left endpoint} ($0\in \Lambda^\max$)]
By \Cref{cor:not:faithful}, let $\gamma^{+}\in K^{+}=\ker(\rho^{+})$ be any hyperbolic element. Let $\phi\in \mathrm{Aff}(X,\omega)$
be a pseudo-Anosov map with $D\phi=\gamma^{+}$. The maximal eigenvalues of the action
of $\phi^\ast$ on $E^{++}$ and $E^{+}$ are $\theta_1$ and $1$, respectively
(since $\gamma^{+}$ belongs to $\ker(\rho^{+})$, the action of $\phi^\ast$ restricted to $E^+$ is trivial).
By \cref{eq:pA}, the top Lyapunov exponent of the KZ cocycle restricted to $E^{+}$
associated to this periodic orbit is simply
$$
\lambda^{+} = \frac{\log(1)}{\log(|\theta_1|)} = 0.
$$
Similarly for any hyperbolic element $\gamma^{-}\in K^{-}=\ker(\rho^{-})$, one has $\lambda^{-}=0$.

Now, since $\Gamma$ is Fuchsian, any non-trivial normal subgroup has the same limit set (see \eg~\cite[Lemma~5.4]{Hopper:Weiss}).
Thus, it is enough to prove that $K\coloneqq\ker(\rho)=K^{+}\cap K^{-}$ is non-trivial.
Since in any Fuchsian group, the set of points in the boundary of $\HH$ fixed by an element of the group is dense in the limit set, we can find elements $\gamma^{\pm} \in K^{\pm}$ that fix different points in the limit set. In particular, $\gamma^{+}$ and $\gamma^{-}$ do not commute. It follows that $\mathrm{id} \neq [\gamma^{+},\gamma^{-}] \in K^{+} \cap K^{-}$ and $K$ is not trivial (\cf~\cite[Theorem~2]{Pardo:remark}). Hence, by the previous discussion, $(0,0)\in \Lambda$ and $0\in \Lambda^\max$.

\item[\hypertarget{right-end}{Right endpoint} ($1\in \overline{\Lambda^\max}$)]
For a matrix $\left(\begin{smallmatrix} a & b \\ c & d \end{smallmatrix}\right) \in \Gamma$, we denote by $\left(\begin{smallmatrix} a^{\pm} & b^{\pm} \\ c^{\pm} & d^{\pm} \end{smallmatrix}\right) \in \rho^{\pm}(\Gamma)$ its image by the representation $\rho^{\pm}$.
Let $\Gamma_0$ be the finite index subgroup introduced in \Cref{sec:interval}.
By \Cref{c:parabolic-in-complement}, we can take $\gamma\in \Gamma$ a parabolic element in $\Gamma_0$ that is not in either of the kernels $K^{\pm}$.
Up to conjugacy, we can assume that $\gamma=\left(\begin{smallmatrix} 1 & t \\ 0 & 1 \end{smallmatrix}\right)$ and $\rho^{\pm}(\gamma)=\left(\begin{smallmatrix} 1 & t^{\pm} \\ 0 & 1 \end{smallmatrix}\right)$, with $t,t^{\pm}\neq 0$. Now let $\gamma'=\left(\begin{smallmatrix} a & b \\ c & d \end{smallmatrix}\right) \in \Gamma_0$ hyperbolic such that $\rho^{\pm}(\gamma')$ is also hyperbolic. Since $\Gamma_0$ is discrete, $\gamma$ and $\gamma'$ do not fix the same point at infinity, and the same is true for $\rho^{\pm}(\gamma)$ and $\rho^{\pm}(\gamma')$ (indeed, the images of $\Gamma_0$ by $\rho^{\pm}$ belong to $\mathrm{Aut}(E^{\pm}(\ZZ))= \mathrm{SL}(2,\ZZ)$ so they are discrete).
It follows that $c,c^{\pm}\neq 0$.

From the equality $\tr(\gamma^n\gamma')=ntc+\tr(\gamma')$, we deduce that $\gamma^n\gamma'$ is hyperbolic for $n$ large enough.
In particular, the top Lyapunov exponent (on the bundle $E^{\pm} \subset H^1(X,\mathbb R)$) associated to this periodic orbit is given by \cref{eq:pA}:
$$
\lambda^{\pm}_n = \frac{\log(nt^{\pm}c^{\pm}+\tr(\rho^{\pm}(\gamma')))}{\log(ntc+\tr(\gamma'))} + o(1) = 1+ o(1).
$$
Thus, $(\lambda^{-}_n,\lambda^{+}_n) = (1,1) + o(1) \in \Lambda$.
In particular,  $1+ o(1) \in \Lambda^\max$ and hence $1 \in \overline{\Lambda^\max}$.
\end{description}

It follows that $\Lambda^\max = [0,1) \subset \left\{\delta_\theta; \ \theta\in \mathcal D \right\}$.

Now, recall that $\mathcal D \subset \mathbb S^1$ is the full measure set of directions $\theta$ such that $\delta_\theta(x)$ is constant for almost every $x \in T(a, b)$.

To conclude, we need to show that $1$ belongs to $\left\{\delta_\theta; \ \theta\in \mathcal D \right\}$.
But this follows from \Cref{c:strip-decomposition}.
In fact, given a strip decomposition in a direction $\theta \in \mathbb S^1$, it is clear that $\theta\in \mathcal D$ and $\delta_\theta = 1$ as any trajectory inside a strip scapes linearly.
\end{proof}

\section{Joint diffusion}
\label{sec:square}

In this section we prove \Cref{thm:square} exhibiting an infinite family of wind-tree billiards
for which the Lyapunov spectrum contains the full square $(0,1)^2$
and thus, that exhibit all possible joint diffusion rates.
To our knowledge, this is a phenomenon that has not been previously exhibited, even for general Fuchsian groups.

In order to prove \Cref{thm:square}, we use the following.
\begin{Proposition}
\label{banane:3}
Suppose that there are parabolic elements in $\ker(\rho^{+}) \setminus \ker(\rho^{-})$ and in $\ker(\rho^{-}) \setminus \ker(\rho^{+})$.
Then, the Lyapunov spectrum $\Lambda$ contains $(0,1)^2$.
\end{Proposition}
\begin{proof}
In the proof of \Cref{thm:main}, we showed that $(0,0)\in \Lambda$ and $(1,1) \in \overline{\Lambda}$.
In order to prove that $(1,0)\in \overline{\Lambda}$, we follow the same strategy of the proof of \Cref{thm:main} (\hyperlink{right-end}{right endpoint}), but taking now $\gamma'\in \ker(\rho^{-})$ hyperbolic element with $\rho^{+}(\gamma')$ hyperbolic instead.
This is possible because $\ker(\rho^{-}) \setminus \ker(\rho^{+}) \neq \emptyset$.

More precisely, let $\gamma\in \ker(\rho^{-}) \setminus \ker(\rho^{+})$ be a parabolic element and $\gamma'\in \ker(\rho^{-}) \setminus \ker(\rho^{+})$ be a hyperbolic element such that $\rho^{+}(\gamma')$ is also hyperbolic.
Then, for $n$ large enough, the element $\gamma_n = \gamma^n\gamma'$ is hyperbolic and $\lambda^{+}_n = 1+ o(1)$, as in the proof of \Cref{thm:main} (\hyperlink{right-end}{right endpoint}).
However, in this case $\gamma_n \in \ker(\rho^{-})$ and therefore $\lambda_n^{-}=0$.
It follows that $(1,0)\in \overline{\Lambda}$.
And, by symmetry, $(0,1)\in \overline{\Lambda}$ as well.

By \Cref{rk:top}, $\Lambda \subset [0,1)^2$.
Then, by \Cref{t.main2}, $\Lambda \subset [0,1)^2$ has convex and dense interior and $\overline{\Lambda}$ contains the four corners of the square $[0,1]^2$.
In particular, $\{(0,0)\}\cup(0,1)^2 \subset \Lambda$.
\end{proof}

Hence, in order to prove \Cref{thm:square}, it is enough to show that $\ker(\rho^{+}) \setminus \ker(\rho^{-})$ and $\ker(\rho^{-}) \setminus \ker(\rho^{+})$ are non-empty and that there are parabolic elements in both differences.

By \Cref{thm:geom}, in order to classify parabolic elements in $K^{\pm} = \ker(\rho^{\pm})$, it is enough to understand one-cylinder directions in $L = L_{a,b} \in \H(2)$.

Recall that we are interested in parameters $(a,b)$ in the set
\[\esq = \left\{ \left(\frac{p}{q},\frac{r}{s}\right) \in (0,1)^2;\, \gcd(p,q) = \gcd(r,s) = 1, \ p,q,r,s \in 2\N-1\right\},\]
from \Cref{thm:square}.

As in the work of Hubert--Lelièvre--Troubetzkoy~\cite{Hubert:Lelievre:Troubetzkoy},
we first rescale $L$ in such a way that it is a surface tiled by $1\times 1$ squares, and distinguish the regular Weierstrass points $A,B,C,E,F$ by their projection into the squares, as in \Cref{fig:weven}.
In our particular case of $(a,b) \in \esq$, if $N \in \N$ is the number of squares, then $N = qs - pr$ is even and $N \geq 8$, where $a = p/q$ and $b = r/s$ are in lowest terms.
Moreover, the projection of the regular Weierstrass points $A,B,C,E,F$ into the unit squares is always as follows:
\begin{itemize}
\item $A$ projects to a corner,
\item $B$ and $E$ project to the center of a horizontal side, and
\item $C$ and $F$ project to the center of a vertical side.
\end{itemize}

\begin{figure}[ht]
\includegraphics[scale=0.5]{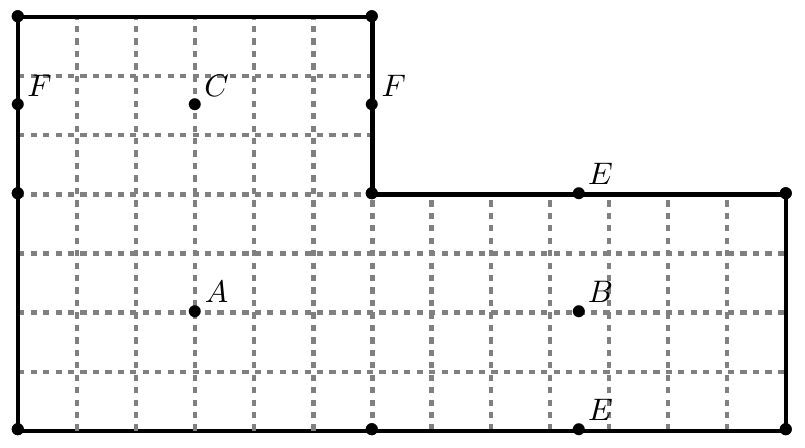}
\caption{The surface $L(3/7,7/15)$ after rescaling and its regular Weiestrass points $A,B,C,E,F$.}
\label{fig:weven}
\end{figure}

Now, given a one-cylinder decomposition of such a surface $L$, it can be rotated, rescaled and sheared in a way to obtain a horizontal cylinder of height one and length $N$, as in \Cref{fig:1cyl}.
Since $N$ is even, both Weierstrass points in the core curve projects to the same point in the unit square.
Thus, there are only two possible pairs of Weierstrass points in the same one-cylinder core curve, namely, $(B, E)$ or $(C,F)$, as shows the preceding list.

\begin{figure}[ht]
\includegraphics{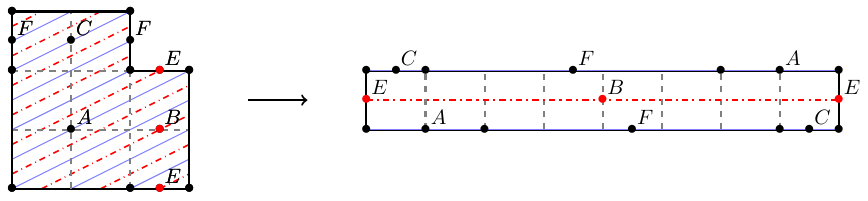}
\caption{A one-cylinder from $L(2/3,2/3)$ in direction with slope $1/2$, after rotating, rescaling and shearing.}
\label{fig:1cyl}
\end{figure}

By our geometric criterion of \Cref{thm:geom}, the previous remark proves that $K^{+}\mathrel{\triangle} K^{-} \neq \emptyset$, since there are always one-cylinder directions in the square-tiled case (see \cite[Proposition~5.1]{Hubert:Lelievre} or \cite[Corolary~A.2]{McMullen:discriminant-spin}).
However, in order to have elements in both differences we have to show that both combinations, $(B,E)$ and $(C,F)$, occur indeed.
To do this, we need the following.

\begin{Lemma}
\label{lemm:perm}
For every $(a,b) \in \esq$, there is $\phi \in \Aff(L_{a,b})$ such that its restriction to the regular Weierstrass points corresponds to the permutation $(B,C)(E,F)$.
\end{Lemma}

\begin{proof}
Let $(a,b) \in \esq$ and $L = L_{a,b}$.
By \Cref{thm:cylWei}, there is a cylinder in $L$ such that its core curve passes through $B$ and $C$.
Since $(a,b) \in \esq$, this cylinder is part of a two-cylinder decomposition, call $X$ and $Y$ the regular Weierstrass points in the core curve of the second cylinder, $X,Y \in \{A,E,F\}$.

Let $T \in \mathrm{Aff}(L)$ be the (primitive) affine multi-twist corresponding to this two-cylinder decomposition.
Note that an affine transformation sends regular Weierstrass points to regular Weierstrass points and, in the primitive square-tiled case, respecting its projections into the square torus.
Moreover, $T$ does not fix all the regular Weierstrass points.
In fact, the Weierstrass points in the core curves of the corresponding cylinders are interchanged if and only if the curve appears an odd number of times in the affine multi-twist.
But the number of times that each core curve appears in the multi-twist are coprime.
So both cannot be even and at least one pair is exchanged.

It follows that $T$ transposes at least one of the pairs $(B,C)$ or $(X,Y)$.
\begin{itemize}
\item
If $T$ interchanges $B$ with $C$, then it has to interchange $E$ with $F$.
In fact, at the level of the square torus, it interchanges the center of a horizontal side with the center of a vertical side.
In other words, any affine transformation exchanges $B$ with $C$ if and only if it exchanges $E$ with $F$.
\item
If $T$ interchanges $X$ with $Y$, then necessarily $\{X,Y\} = \{E,F\}$.
In fact, $A$ is fixed by any affine transformation since it is the only integral regular Weierstrass point.
Thus, $T$ interchanges $E$ with $F$, and, by the same previous argument, it has to interchange $B$ with $C$.
\end{itemize}
It follows that $\{X,Y\} = \{E,F\}$ and $T$ interchanges both pairs $B,C$ and $E,F$.
\end{proof}

\begin{Remark}
Using similar ideas, Gutiérrez-Romo and the fourth named author give in \cite{Gutierrez-Romo:Pardo} a complete description of the action of the affine group on Weierstrass points for every Veech surface in $\H(2)$.
\end{Remark}

\begin{Corollary}
\label{coro:one-cyl}
For every $(a,b) \in \esq$, there are one-cylinder decompositions of $L_{a,b}$ such that the core curve of the cylinder passes through $B$ and $E$ \emph{(resp. $C$ and $F$)}.
\end{Corollary}

\begin{proof}
Since $L = L_{a,b}$ is a square-tiled surface, it has a one-cylinder decomposition (see \cite[Proposition~5.1]{Hubert:Lelievre} and \cite[Corolary~A.2]{McMullen:discriminant-spin}).
It follows that its core curve passes through either $B$ and $E$, or $C$ and $F$.
If it passes through $B$ and $E$, then we can apply $\phi \in \Aff(L)$ from \Cref{lemm:perm} to obtain a one-cylinder decomposition of $L$ whose core curve passes through $C$ and $F$, or viceversa.
\end{proof}

Now, we are ready to prove \Cref{thm:square}.

\begin{proof}[Proof of \Cref{thm:square}]
By \Cref{coro:one-cyl}, there is a one-cylinder decomposition in $L = L_{a,b}$ such that the core curve $\gamma$ of the cylinder passes through $B$ and $E$.
But then, by \Cref{thm:geom}, the corresponding parabolic element belongs to $K^{+}$ since $E \in \gamma$, but it does not belong to $K^{-}$ since $F \notin \gamma$. Thus, $K^{+} \not\subset K^{-}$. Analogously, $K^{-} \not\subset K^{+}$, and both differences contain parabolic elements.
By \Cref{banane:3}, it follows that the Lyapunov spectrum contains the full square $(0,1)^2$ and so, by Remark~\ref{rk:join:diffusion}, it is contained in the set of joint diffusion rates.
\end{proof}

\appendix

\section{Weak convexity of the set of ergodic measures}
Sigmund has proved~\cite{Sigmund} that the space of ergodic measures on a transitive subshift of finite type
is arcwise connected with respect to the weak-$\ast$ topology. This can be generalized as follows.
The proof is standard, but we could not find any reference for that.

\begin{Theorem}\label{t.convexity}
Let $\Sigma$ be a transitive subshift of finite type and
$\varphi_1,\dots,\varphi_d\colon \Sigma\to \RR$ be continuous functions.
Then for any periodic orbits $O_1,\dots,O_\ell$
and $\alpha_1,\dots,\alpha_\ell> 0$ satisfying $\sum_k \alpha_k=1$,
there exists an ergodic measure $m$ such that $\int \varphi_id m=\sum_k \alpha_k \int\varphi_i d\nu(O_k)$ for each $i=1,\dots,d$,
where $\nu(O_k)$ denotes the invariant probability measure supported on $O_k$.
\end{Theorem}
\begin{proof}
We build inductively a sequence of $\ell$-uples of periodic orbits
$(O_1^n,\dots,O_\ell^n)$ such that the simplex in $\RR^d$
with vertices $(\int\varphi_1d \nu(O^n_k),\dots,\int\varphi_d d\nu(O^n_k))$ for $1\leq k \leq \ell$ decreases towards
the point $(\sum_k \alpha_k \int\varphi_1 d\nu(O_k),\dots,\sum_k \alpha_k \int\varphi_d d\nu(O_k))$.
This is done as follows.
We choose small neighborhood $U_1,\dots,U_\ell$
of the periodic orbits $O_1^n,\dots,O_\ell^n$.
Given numbers $b_1,\dots,b_\ell$,
the specification property of the subshift allows to build a periodic orbit $O'$ with arbitrarily large period
which spends a proportion of time arbitrarily close to $b_i$ in $U_i$, for each $1\leq i\leq \ell$.
By adjusting the values of the $b_i$, the values $(\int\varphi_1d \nu(O'),\dots,\int\varphi_d d\nu(O'))$ in $\RR^d$
can be chosen arbitrarily close any point inside the simplex with vertices
$(\int\varphi_1d \nu(O^n_k),\dots,\int\varphi_d d\nu(O^n_k))$.
One can in this way build $\ell$ periodic orbits $O_1^{n+1},\dots,O_\ell^{n+1}$
as wanted.

Note that at each step $n$ of the construction,
the $\ell$ measures $\nu(O^{n+1}_k)$ can be chosen arbitrarily close together for the weak-$\ast$
topology. More precisely, let $\mathcal{F}=(\psi_j)$ be a countable collection
which is dense in the space of continuous functions from $\Sigma$ to $\RR$.
In the previous construction, one can require that for each function $\psi_j$,
the values $\int \psi_j d\nu(O^n_k)$ converge as $n\to +\infty$
to a number which does not depend on $k$.
This implies that for each $k$ the sequence $\nu(O^n_k)$
converges towards an invariant probability measure $m$ which does not depend on $k$.
Moreover $\int \varphi_id m=\sum_k \alpha_k \int\varphi_i d\nu(O_k)$ for each $i=1,\dots,d$,
as required.

It remains to check that $m$ can be chosen ergodic:
we prove that for any $\psi_i\in \mathcal{F}$ and $q\geq 1$,
there exists $L$ large such that for $m$-almost every point $x$,
the next property $\mathcal{P}(x,L,i,q)$ holds:
\begin{equation}\tag{$\mathcal{P}(x,L,i,q)$}
\left|[\psi_i(x)+\psi_i(\sigma(x))+\dots+\psi_i(\sigma^{L-1}(x))]/L-\int\psi_i d m\right|<1/q.
\end{equation}
By construction, for $n$ large, any periodic orbit $O^n_k$ satisfies
$|\int \psi_i d\nu(O^n_k)-\int\psi_i d m|<1/(4q)$.
Hence there exists $L_0$ large such that for any $x\in \cup_k O^n_k$ the property
$\mathcal{P}(x,L_0,i,3q)$ holds.

Since the periodic orbits $O^n_k$ induce measures that are close,
they contain points that are close. One can thus find finite words
$w^n_k$ such that $\overline {w^n_k}\in O^n_k$ for each $k$,
and such that all the $w^n_k$ have the same initial symbol.
Periodic points in $O^{n+1}_k$ may be built as
periodic words of the form $\overline{(w^n_1)^{\beta^n_1}(w^n_2)^{\beta^n_2}\cdots(w^n_\ell)^{\beta^n_\ell}}$.
Consequently, choosing $L\gg L_0$ large, and the integers $\beta^n_i$ much larger,
for any point $x$ in an orbit $O^{n+1}_k$, the piece of orbit
$x, \sigma(x),\dots,\sigma^{L-1}(x)$ decomposes into pieces of orbits of length $L_0$
arbitrarily close to pieces of orbits inside the $O^{n}_k$
and a subset of iterates whose cardinal is small relative to $L$.
Then, by continuity of $\psi_i$, the property $\mathcal{P}(x,L,i,2q)$ holds.
One builds the subsequent orbits $O^{n'}_k$ in the same way,
hence the property $\mathcal{P}(x,L,i,2q)$ holds for any point in any orbit
$O^{n'}_k$ for $n'$ large enough, so that taking the limit the property
$\mathcal{P}(x,L,i,q)$ holds for any point $x$ in the support of $\mu$.

At each step of the construction, one can require the $\beta^n_i$ to be large enough
and control a larger number of conditions $\mathcal{P}(x,L,i,q)$.
Since the set of pairs $(\psi_i,q)$ is countable, one deduces that for any $(i,q)$,
there exists $L\geq 1$ such that $\mathcal{P}(x,L,i,q)$ holds on any point of the support of $m$,
concluding the ergodicity.
\end{proof}

\printbibliography

\end{document}